%% file: GokovaMorseArxiv.tex
\documentclass[11pt]{amsart}

\usepackage{amsmath, amsthm, amsfonts, amssymb, amscd, flafter, graphicx, verbatim, rotating, pkmacros, url}

\usepackage[section]{placeins}

\input{macros}

\newcommand\cv{\mathcal{V}}

\newcommand\ch{\mathcal{H}}
\newcommand\ce{\mathcal{E}}
\newcommand\ve{\varepsilon}
\newcommand\cp{\mathcal{P}}
\newcommand\ca{\mathcal{A}}

\newcommand\tcv{\widetilde{\mathcal{V}}}
\newcommand\tcg{\widetilde{\mathcal{G}}}
\newcommand\tce{\widetilde{\mathcal{E}}}
\newcommand\tch{\widetilde{\mathcal{H}}}
\newcommand\tve{\tilde\varepsilon}
\newcommand\tcp{\widetilde{\mathcal{P}}}
\newcommand\tca{\widetilde{\mathcal{A}}}
\newcommand\tr{\tilde{\rho}}
\newcommand\br{\bar{\rho}}

\newcommand\bcv{\overline{\mathcal{V}}}

\newcommand\bce{\overline{\mathcal{E}}}
\newcommand\bch{\overline{\mathcal{H}}}
\newcommand\bve{\bar\varepsilon}
\newcommand\bcp{\overline{\mathcal{P}}}
\newcommand\bca{\overline{\mathcal{A}}}

\title[The combinatorics of {M}orse theory with boundary]{The combinatorics of {M}orse theory with boundary}

\author{Jonathan M.\ Bloom}
\thanks{The author was supported by NSF grant DMS-1104032.}

\begin{document}

\begin{abstract}
We prove several combinatorial results on path algebras over discrete structures related to directed graphs.  These results are motivated by Morse theory on a manifold with boundary and, more generally, by Floer theory on a configuration space with boundary.  Their purpose is to organize cobordism relationships among moduli spaces in order to define new algebraic invariants.  We discuss applications to the Morse and Fukaya categories, and to work with John Baldwin on a bordered monopole Floer theory.
\end{abstract}

\maketitle

\tableofcontents

\section{Introduction}

In this article, we prove several combinatorial results on path algebras over discrete structures related to directed graphs.  These results are motivated by Morse theory on a manifold with boundary and, more generally, by Floer theory on a configuration space with boundary.  Their purpose is to organize cobordism relationships among moduli spaces in order to define new algebraic invariants.  Sections \ref{sec:higraphs} through \ref{sec:bigraphs} give a fully self-contained exposition of the mathematical content.  The Appendix briskly reviews Morse homology on a manifold with boundary, as defined by Kronheimer and Mrowka as the model for monopole Floer homology \cite{km}.  We also explain there the connection to cell structures and CW-homology.  In the following three subsections of the introduction, we indicate the flavor of the combinatorics, its relation to Morse theory, and applications to Morse and Floer theory with boundary.

\subsection{Counting paths}

Let $\mathbb{F}$ be the 2-element field.  To a transitive digraph $G$, one can associate a differential graded algebra $\mathcal{A}$ over $\mathbb{F}$, the {\em path DGA} of $G$.  The $\mathbb{F}$-basis of $\mathcal{A}$ is given by paths in $G$, the product by concatenation, and the grading by path length.  The differential $\delta$ of the edge $(v_1, v_2)$ is the sum of all length-two paths from $v_1$ to $v_2$, and $\delta$ extends to paths by the Leibniz rule.  The identity $\delta^2 = 0$ follows from the fact that a sequence of three compatible edges can be concatenated in two orders: $(ab)c$ and $a(bc)$.  The element $D \in \mathcal{A}$ given by the sum of all the edges of $G$ satisfies the {\em structure equation}
$$\delta D = D \circ D$$
as both sides equal the sum of all length-two paths in $G$.  The digraph corresponding to Morse homology on a closed manifold is shown in Figure \ref{fig:morseInt}.

In a directed hypergraph, the source and target of an edge are each subsets of the set of vertices.  In Section \ref{sec:higraphs}, we introduce a class of directed hypergraphs called {\em higraphs} on which the path DGA is well-defined and the structure equation continues to hold.  The higraph underlying the cup product is shown in Figure \ref{fig:muTwo}.

\begin{figure}[htp]
\centering
\includegraphics[width=156mm]{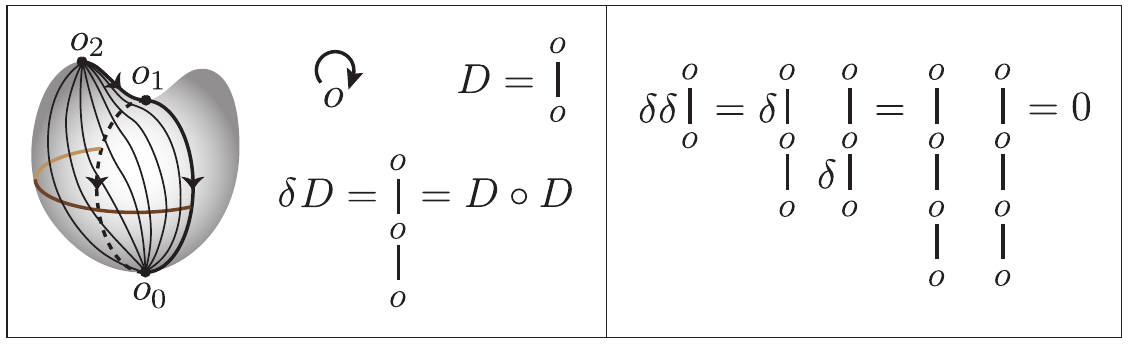}
\caption{The higraph (and digraph) consisting of a single vertex and a single edge corresponds to Morse homology on a closed manifold.  The vertex $o$ represents the Morse complex and the edge $D$ represents the Morse differential.  The differential $\delta$ on the path DGA encodes the key cobordism relation between trajectory spaces illustrated on the surface at left.  The structure equation is immediate.  At right, the identity $\delta^2 = 0$ encodes the fact that each moduli space of twice-broken trajectories bounds two moduli spaces of once-broken trajectories.}
\label{fig:morseInt}
\end{figure}

\begin{figure}[htp]
\centering
\includegraphics[width=156mm]{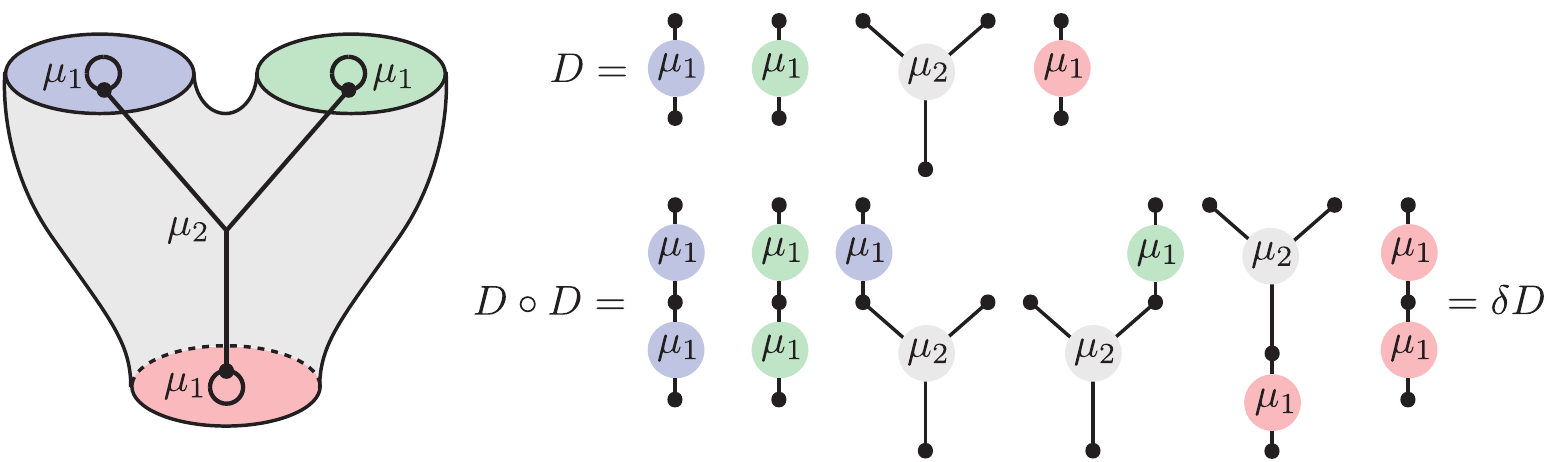}
\caption{The higraph at left consists of three vertices and four edges.  The three edges labeled by $\mu_1$ are loops and the fourth edge $\mu_2$ runs from the two top vertices to the bottom vertex.  In the context of the Morse category discussed in Section \ref{sec:floer}, the structure equation implies that $\mu_1$ is a differential and $\mu_2$ is a chain map,  inducing the cup product structure.}
\label{fig:muTwo}
\end{figure}

In Sections \ref{sec:abstractpaths} and \ref{sec:bigraphs} we extend these results to our main character: the {\em bigraph} of a higraph (the ``b'' stands for blow-up and boundary).  Each higraph vertex is covered by three bigraph vertices called interior ($o$), stable ($s$), and unstable ($u$), and each higraph edge is covered by a number of bigraph edges called interior and boundary.  The grading on the path DGA $\mathcal{\widetilde{A}}$ of a bigraph is defined in terms of weight rather than length.  The weight of an interior edge is 1 and the weight of a boundary edge depends on the number and types of vertices in its source and target (Definition \ref{def:weight}).  The differential $\tilde \delta$ on a weight $w$ interior edge (resp., boundary edge) is the sum all weight $w+1$ paths (resp., boundary paths) with the same source and target.  Our first main results (Theorem \ref{thm:deltatildesq}) verifies that $\tilde \delta^2 = 0$.  A bigraph is {\em tree-like} if every edge has singleton target and {\em rake-like} if every edge has singleton source.  For tree-like bigraphs, our second result (Theorem \ref{thm:bistructure}) states that the element $\tilde D \in \mathcal{\widetilde A}$ given by the sum of all weight 1 paths between interior and unstable vertices satisfies the structure equation
$$\tilde\delta \tilde D = \tilde D \circ \tilde D.$$
Dually, for rake-like bigraphs, the structure equation holds when $\tilde D$ is given by the sum of all weight 1 paths between interior and stable vertices.

The bigraph covering the higraph in Figure \ref{fig:morseInt} is shown at the top of Figure \ref{fig:morsePaths}.  The reader may look to the Appendix, particularly Figure \ref{fig:morseOps}, to see how this fundamental example encodes the structure of Morse homology on a manifold with boundary.  There are three vertices, four interior edges in black, and four boundary edges in red.  The dashed red edge from $s$ to $u$ has weight $0$ and the doubled red edge from $u$ to $s$ has weight $2$.  The other six edges have weight 1.  The element $D$ is the sum of all weight 1 paths starting or ending at $o$ or $s$. The middle portion of Figure \ref{fig:morsePaths} shows the value of the differential $\tilde\delta$ on each edge, together with a surface illustrating the cobordism relation between trajectory spaces encoded by $\tilde\delta(o,o)$.  The bottom portion depicts the cancellation of three pairs of paths in $\tilde\delta \tilde D$ which do not appear in $\tilde D \circ \tilde D$.  The greyscale surfaces along the bottom represent boundary components of 3-manifolds, in order to indicate the relationship of this cancellation to Morse theory.  The families of red trajectories lie on the boundary while the families of black trajectories flow into the blue interior.  The cancellation reflects the fact that each highlighted broken trajectory arises as the end of a one-parameter family of (broken) trajectories in two ways.

The bigraph in Figure \ref{fig:pantsgraph} covers the higraph in Figure \ref{fig:muTwo}.  Now $\tilde \delta \tilde D$ expands out to 197 paths, 62 of which cancel in pairs, leaving the 135 terms of $\tilde D \circ \tilde D$.  A careful analysis of this cancellation phenomenon underlies the proof of Theorem \ref{thm:bistructure}.

\begin{figure}[htp]
\centering
\includegraphics[width=156mm]{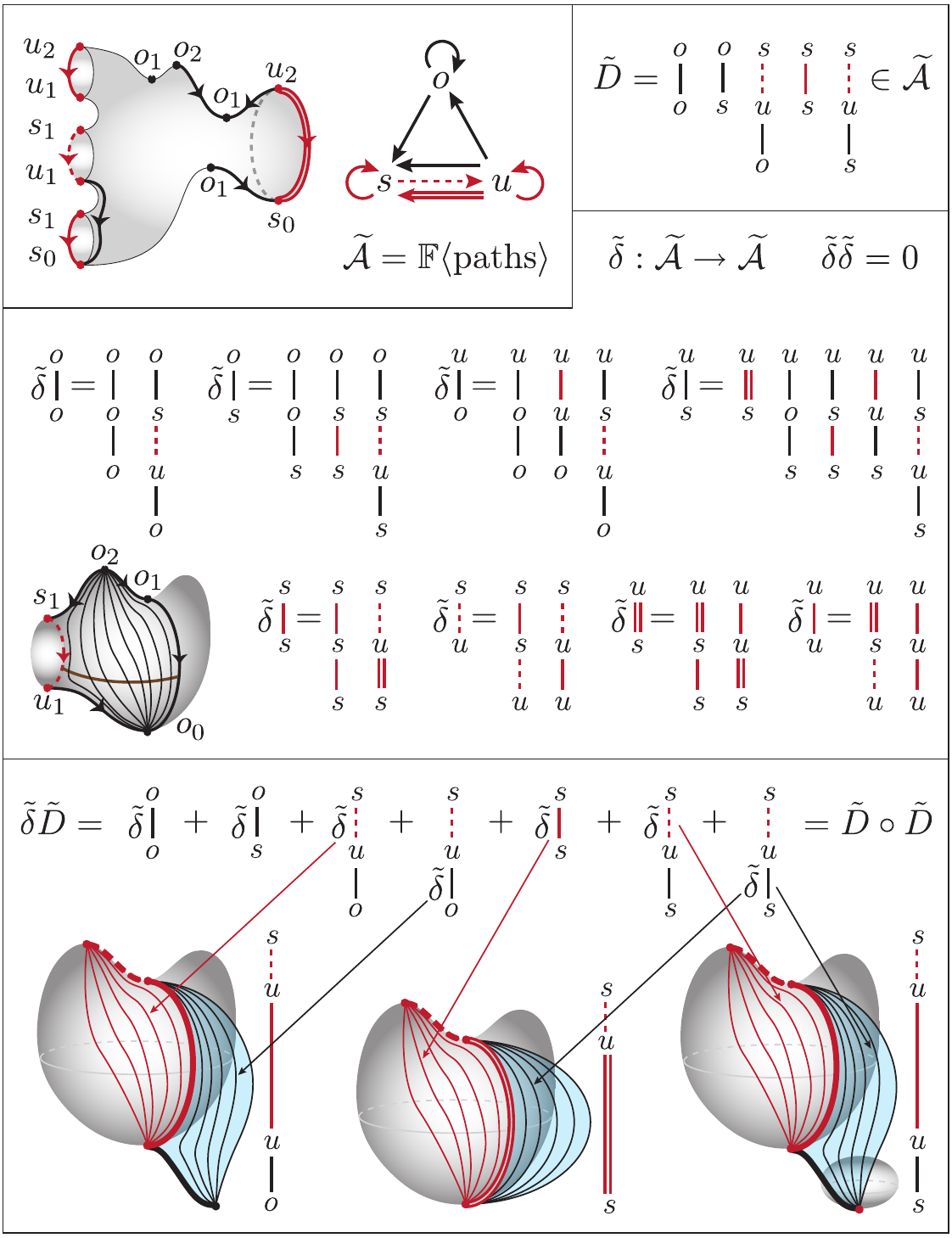}
\caption{The path DGA of the bigraph representing Morse homology on a manifold with boundary (compare with Figure \ref{fig:morseOps}).}
\label{fig:morsePaths}
\end{figure}

\subsection{The view from Morse theory}
The Morse homology package extracts algebraic invariants of smooth manifolds from cobordism relationships between moduli spaces of gradient trajectories.  The fundamental cobordism relation takes the schematic form
\begin{align}
\label{schematic}
\partial X  = X \circ X,
\end{align}
in which the boundary of one moduli space is identified with products of others moduli spaces.  More precisely, we have
\begin{align}
\label{xab}
\partial \overline{X}^a_b = \bigcup_c \ X^a_c \times X^c_b.
\end{align}
where $X^a_b$ is the moduli space of unparameterized trajectories from $a$ to $b$, and $\overline{X}^a_b$ denotes a compactification of $X^a_b$. One defines an endomorphism $D$ on the vector space
generated by the set of critical points by setting
$$\langle D a, b \rangle = |X^a_b|.$$
One may also define an endomorphism $\delta D$ by setting
$$\langle (\delta D) a, b \rangle = |\partial \overline{X}^a_b|.$$
The cobordism relation \eqref{xab} implies the {\em structure equation}
\begin{align}
\label{structure}
\delta D = D \circ D,
\end{align}
a precise algebraic translation of the schematic \eqref{schematic}.  Furthermore, $\delta D$ is the zero map, since $\partial \overline{X}^a_b$ is finite if and only if it is the boundary of a compact 1-manifold and hence an even number of points.  We conclude from the structure equation that $D$ is indeed a differential.  Note that the same framework is used to prove that Morse homology is invariant and functorial.
Going one step further, one may define an endomorphism $\delta\delta D$ by setting
$$\langle (\delta\delta D) a, b \rangle = |\partial(\partial \overline{X}^a_b)|.$$
This is the zero map for a different topological reason: the compactification of $X^a_b$ may be given the structure of a smooth manifold with corners.

\begin{quotation}
{\em This viewpoint naturally extends to the case of a manifold with boundary.}
\end{quotation}

For a manifold with boundary, the space $X$ is itself a union of products of moduli spaces of trajectories (that is, $X$ is composed of several types of possibly-broken trajectories).  The space $\overline{X}$ is built from the compactifications of the factors, its boundary given by the Leibniz rule.
The structure equation continues to hold, although now after the cancellation of several pairs of terms in $\delta D$ which do not appear in $D \circ D$ as illustrated in Figure \ref{fig:morsePaths}.  And the maps $\delta D$ and $\delta\delta D$ still vanish.  We can summarize the situation as follows: without boundary, Morse homology depends on geometric analysis and a bit of topology;  with boundary, Morse homology depends on (more involved) geometric analysis, a bit of topology, and the simplest case of the combinatorics that is comprehensively developed herein.

\begin{remark}
Kronheimer and Mrowka study Morse theory with boundary in the context of monopole Floer homology.  They do not prove that $\overline{X}$ is a manifold with corners, but rather that it is stratified by smooth manifolds and has an even number of ends in the 1-dimensional case.  The example in Figure \ref{fig:eightPaths} indicates that $\overline{X}$ is not a smooth manifold with corners when its dimension is three or more.  On the other hand, Proposition \ref{prop:transpose} shows that $|\partial(\partial \overline{X})| = 0$, which corresponds to the identity $\delta^2 = 0$ in the path DGA setting and suggests that $\overline{X}$ may be a manifold with corners in the 2-dimensional case.
\end{remark}

\subsection{Morse and Floer theory with boundary}
\label{sec:floer}
Morse homology is an approach to analyzing the topology of a smooth manifold by studying the gradient flow of a generic smooth functional.  Floer theory generalizes this approach to infinite-dimensional configuration spaces arising in symplectic geometry and gauge theory.  Two important examples, Lagrangian intersection Floer homology and monopole Floer homology, may be placed in parallel with finite-dimensional Morse homology as follows:
\begin{itemize}
\item The {\em Morse complex} $\mathit{C}(f)$ of a smooth function $f$ on a closed manifold is generated by the critical points of $f$; the differential counts gradient trajectories between critical points.
\item The {\em Langrangian Floer complex} $\mathit{CF}(L_0,L_1)$ of two Lagrangian submanifolds of a symplectic manifold $X$ is generated by the points of $L_0 \cap L_1$; the differential counts pseudo-holomorphic bigons with edge $i$ on $L_i$.  Here the functional is the symplectic action functional on the space of paths from $L_0$ to $L_1$ \cite{fooo}.
\item The {\em monopole Floer complex} $\widehat{\mathit{C}}(Y)$ of a closed 3-manifold $Y$ is generated by solutions (called monopoles) to the 3-dimensional Seiberg-Witten equations on $Y$; the differential counts solutions to the 4-dimensional Seiberg-Witten equations on $Y \times \mathbb{R}$.  Here the functional is the Chern-Simons-Dirac functional on the configuration space $\mathcal{B}(Y)$ consisting of pairs of certain connections and spinors over $Y$ modulo gauge \cite{km}.
\end{itemize}
We caution that these rough descriptions suppresses auxiliary data (such as metrics, perturbations, and almost-complex structures) and ignore bubbling phenomena that may prevent the Langrangian Floer complex from being well-defined.  Note also that the configuration space $\mathcal{B}(Y)$ has a boundary consisting of reducible configurations, which is why there are three versions of monopole Floer homology.  The correspondence with Morse homology is reflected by the exact sequences
\begin{align*}
\cdots \longrightarrow H_*(\partial M) \longrightarrow H_*(M)  \longrightarrow H_*(M, \partial M) \longrightarrow \cdots  \\
\cdots \longrightarrow \hmbar(Y) \longrightarrow \hmb(Y) \longrightarrow \hmfrom(Y)  \longrightarrow \cdots.
\end{align*}

Singular cohomology admits an associative product structure, the cup product, which may be defined Morse-theoretically using gradient trees (or equivalently, triple intersections of stable and unstable manifolds with respect to several Morse functions).
Fukaya recognized that deeper multiplicative structure is naturally encoded by an $A_\infty$-category, which determines the Massey products and rational homotopy type \cite{fukOh}.
\begin{itemize}
\item The objects of the {\em Morse category} of a closed manifold are smooth functions and $\mathrm{Mor}(f_0,f_1) = \mathit{C}(f_0 - f_1)$.  The multiplication map $$\mu_k : \mathit{C}(f_0 - f_1) \otimes \cdots \otimes \mathit{C}(f_{k-1} - f_k) \to \mathit{C}(f_0 - f_k)$$ counts gradient trees with $k$ inputs and $1$ output.
\end{itemize}
Here $\mu_1$ is the Morse differential and $\mu_2$ is the chain map defined by
$$\langle \mu_2(a\otimes b),c \rangle = |\{x \in M \, : \, \lim_{t \to \infty} \phi^{f_0 - f_1}_t(x) = a, \ \lim_{t \to \infty} \phi^{f_1 - f_2}_t(x) = b, \ \lim_{t \to \infty} \phi^{f_0 - f_2}_t(x) = c\}|$$
where $\phi^f_t: M \to M$ denotes the flow induced by $\nabla f$.  The map $\mu_3$ provides a homotopy between $\mu_2(\mu_2(\cdot,\cdot), \cdot)$ and $\mu_2(\cdot, \mu_2(\cdot,\cdot))$.  In particular, $\mu_2$ induces an associative product on homology, dual to the cup product.  See \cite{seid} for background on $A_\infty$-categories.

One application of our combinatorial results is to extend the definition of the Morse category to manifolds with boundary.  For the appropriate sequence of higraphs, the path DGAs of bigraphs specify the higher multiplications, and the structure equation is equivalent to the $A_\infty$-relations.  For instance, the map $\mu_2$ is defined to count broken gradient trees as specified in Figure \ref{fig:pantsgraph}.  The higraph for $\mu_3$ is shown in Figure \ref{fig:muThree}, and the general case is similar.

Fukaya introduced the Morse category as a model for his eponymous category, which encodes the higher multiplicative structure of Langrangian intersection Floer homology.
\begin{itemize}
\item The objects of the {\em Fukaya category} of a symplectic manifold $X$ are Lagrangian submanifolds and $\mathrm{Mor}(L_0,L_1) = \mathit{CF}(L_0,L_1)$.  The multiplication map $$\mu^\text{Lag}_k : \mathit{CF}(L_0, L_1) \otimes \cdots \otimes \mathit{CF}(L_{k-1}, L_k) \to \mathit{CF}(L_0, L_k)$$ counts pseudo-holomorphic $(k+1)$-gons with edge $i$ on $L_i$.
\end{itemize}
The connection with the Morse category goes as follows.  Given a smooth function $f$ on $M$, the graph of $df$ is a Lagrangian in $T^*M$ with the standard symplectic structure.  The Fukaya category generated by such Lagrangians is $A_\infty$-equivalent to the Morse category of $M$, which implies the Arnold conjecture for the zero-section \cite{fukOh}, \cite{abu1}.

The Fukaya category is a central algebraic structure in symplectic geometry and mirror symmetry.  It is also notoriously difficult to define or compute rigorously \cite{fooo}, \cite{seid}.  See \cite{lekPer} for a state-of-the-art computation: the Fukaya category of the 2-dimensional torus. A divide-and-conquer approach to computation might go as follows: chop the symplectic manifold into simpler pieces and reassemble its Fukaya category by algebraically compiling the data of some notion of Fukaya category for each piece.  Akaho has taken a step in this direction, defining the Floer complex $\mathit{CF}(L_1,L_2)$ of two Lagrangians in an open symplectic manifold with concave end \cite{akaho}.  A slice of the end inherits a contact structure and intersects the Lagrangians in Legendrians.  The complex  $\mathit{CF}(L_1,L_2)$ is generated by both points of $L_1 \cap L_2$ and Reeb chords between the Legendrians.
The Reeb chords play the role of boundary critical points and the differential takes the same form as the Morse differential on a manifold with boundary.   As such, our bigraph framework should specify the higher multiplications needed to define the Fukaya category of an open symplectic manifold with concave end as well.

Our primary application, in joint work with John Baldwin, is to define a gauge-theoretic analogue of the Fukaya category using monopole Floer theory.  This monopole category combines older ideas of Segal, Donaldson, and Fukaya \cite{fuk1} with the comprehensive analysis of Kronheimer and Mrowka \cite{km}.
\begin{itemize}
\item The objects of the {\em monopole category} of a closed surface $\Sigma$ are 3-manifolds with boundary parameterized by $\Sigma$, and $\mathrm{Mor}(Y_0,Y_1) = \widehat{\mathit{C}}(Y_0 \cup_\Sigma -Y_1)$.  The map $$\mu^\text{mon}_k : \widehat{\mathit{C}}(Y_0 \cup -Y_1) \otimes \cdots \otimes \widehat{\mathit{C}}(Y_{k-1} \cup -Y_k) \to \widehat{\mathit{C}}(Y_0 \cup -Y_k)$$ counts monopoles on a 4-dimensional cobordism $$W : (Y_0 \cup -Y_1) \sqcup \cdots \sqcup (Y_{k-1} \cup -Y_k) \to Y_0 \cup -Y_k$$ over a family of metrics parameterized by the $(k-2)$-dimensional associahedron.
\end{itemize}
The cobordism $W$ is constructed by gluing $\partial Y_i \times [0,1] \subset Y_i \times [0,1]$ to $\Sigma \times e_i \subset \Sigma \times \triangle$ for $0 \leq i \leq k$, where $\triangle$ is a polygonal region bounded by $k+1$ line segments alternating with $k+1$ arcs $e_i$.  The case $k=2$ is depicted at right in Figure \ref{fig:compMor}.  In general, the family of metrics interpolates between all ways of cutting $W$ along disjoint hypersurfaces $Y_i \cup -Y_j$ corresponding to chords in $\triangle$ between non-adjacent arcs $e_i$.  Figure \ref{fig:muThree} illustrates the interval of metrics in the case $k = 3$ from a less rigid perspective.

\begin{figure}[tbp]
\centering
\includegraphics[width=156mm]{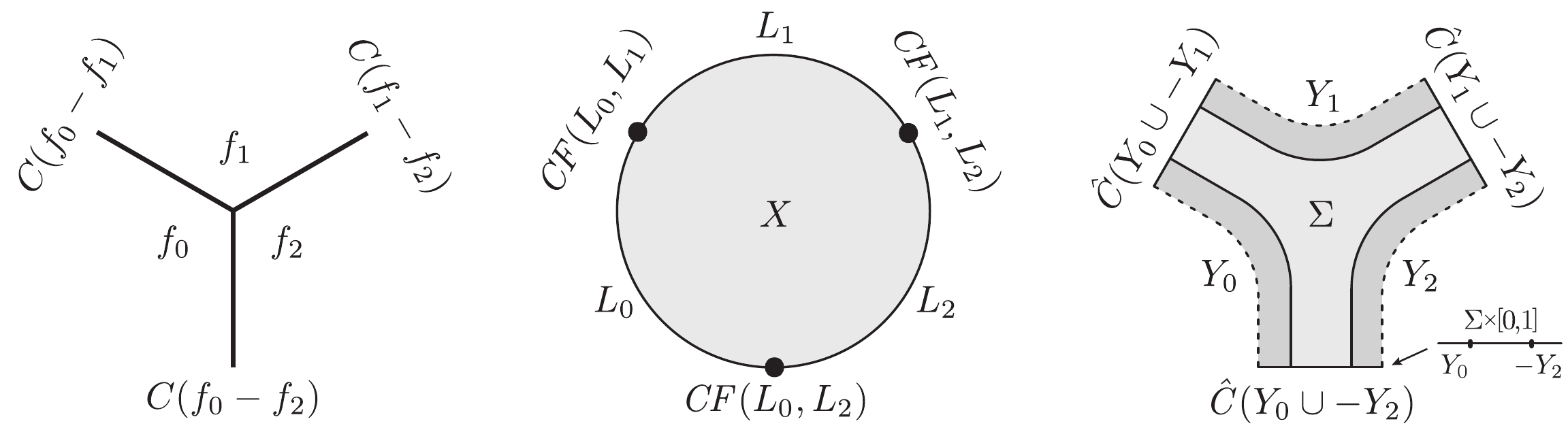}
\caption{The $\mu_2$ multiplication map in the Morse, Fukaya, and monopole $A_\infty$-categories counts gradient trees, pseudo-holomorphic triangles, and monopoles on a $4$-dimensional cobordism, respectively.  On homology, this map induces composition of morphisms in an ordinary category.}
\label{fig:compMor}
\end{figure}

\begin{figure}[htp]
\centering
\includegraphics[width=140mm]{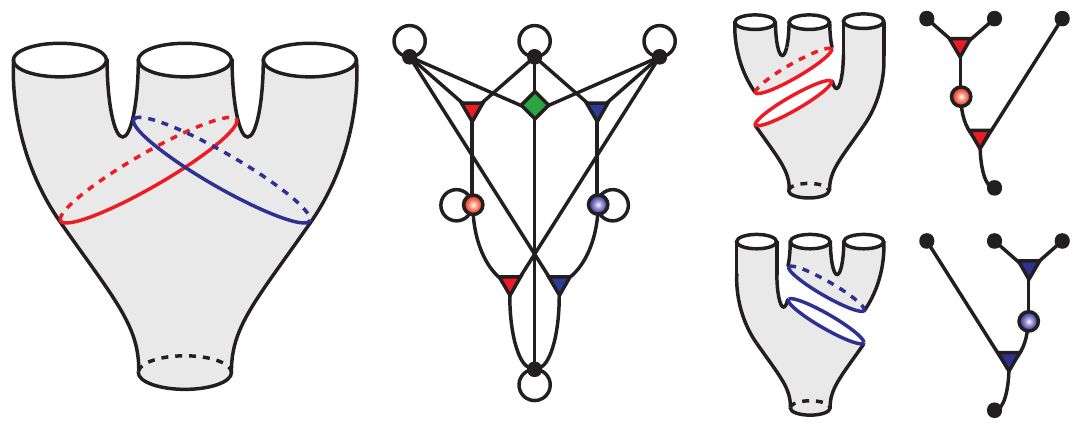}
\caption{The $\mu_3$ multiplication map is represented by the green edge from the three top vertices to the bottom vertex.  The full higraph contains 6 vertices and 11 edges. The compositions $\mu_2(\mu_2(\cdot,\cdot), \cdot)$ and $\mu_2(\cdot, \mu_2(\cdot,\cdot))$ are each represented by paths of length 2 at right.  Due to boundary (reducibles), the terms of the monopole maps $\mu^\text{mon}_1$, $\mu^\text{mon}_2$, $\mu^\text{mon}_3$ are encoded by the corresponding bigraph.  Roughly, $\mu^\text{mon}_3$ counts monopoles on a 4-dimensional cobordism (at left) over a family of metrics parameterized by the 1-dimensional associahedron: the interval $[-1,1]$.  The metric stretches normal to the red (resp., blue) hypersurface as the parameter goes to $-1$ (resp., 1).  The degenerate metrics at $\pm 1$ correspond to composite cobordisms (at right).  The structure equation implies the $A_\infty$-relation for $\mu^\text{mon}_3$ (not shown).}
\label{fig:muThree}
\end{figure}

In light of the isomorphism between monopole and Heegaard Floer homology (see \cite{honda1} and \cite{taubes1}), it is natural to conjecture that the monopole category of a surface $\Sigma$ of genus $g$ is $A_\infty$-equivalent to the subcategory of the Fukaya category of $\mathrm{Sym}^g(\Sigma)$ generated by Heegaard tori. Furthermore, the monopole category fits into the framework of a {\em bordered monopole Floer theory}, a form of extended TQFT.  Roughly, the bordered monopole Floer package is a 2-functor sending surfaces to $A_\infty$-categories, 3-dimensional cobordisms to $A_\infty$-functors, and 4-dimensional cobordisms with corners to $A_\infty$-natural transformations.  In dimensions 2 and 3, this theory is analogous to the version of bordered Heegaard Floer theory developed by Lekili and Perutz \cite{lekPer2} (which in turn is related to the original version \cite{lot1}).  In our case, for each genus, the monopole category is split-generated by an explicit, finite set of parameterized handlebodies through the surgery exact triangle, and the homology algebra is known.  We hope that this theory will point toward an axiomatization of Floer homology in low-dimensional topology.

The preceding applications utilize the sequence of higraphs arising from associahedra encoding homotopy associativity.  Other higraphs are useful as well.  The link surgery spectral sequence is most naturally defined using permutohedra encoding homotopy commutativity; these higraphs have the vertex set of a hypercube with an edge between every ordered pair of vertices \cite{jmb1}.  Baldwin and the author's constructions in bordered monopole Floer theory make use of more general polytopes called graph associahedra \cite{d}, including cyclohedra related to Hochschild homology.  More basically, a tree higraph (one non-loop edge with singleton target) leads to an $\hmfrom$ map for a connected cobordism with disconnected source as in Figure \ref{fig:pantsgraph}.  Dually, a rake higraph (one non-loop edge with singleton source) leads to an $\hmb$ map for a connected cobordism with disconnected target.  As a final example, we have appended the higraph organizing the proof of the monopole surgery exact triangle as Figure \ref{fig:exactTriangle}.  To cover these and future cases, we have aimed for a definition of higraphs that is as general as possible while still satisfying a unified combinatorial theory.

\section{Higraphs }
\label{sec:higraphs}

\begin{definition}
A {\em directed graph}, or {\em digraph}, $G = (V, E)$ consists of a vertex set $V$ and an edge set $E \subset V \times V$.  The edge $e = (v_1,v_2)$ has {\em source} $s(e) = v_1$ and {\em target} $t(e) = v_2$.

A {\em directed hypergraph} $\ch = (\cv, \ce)$ consists of a vertex set $\cv$ and an edge set $\ce \subset 2^\cv \times 2^\cv$.  The edge $\ve = (I,J)$ has {\em source} $s(\ve) = I \subset \cv$ and {\em target} $t(\ve) = J \subset \cv$.
\end{definition}

We reserve the term higraph for a special kind of directed hypergraph.

\begin{definition}
A {\em higraph} $\ch$ is a directed hypergraph $(\cv, \ce)$ such that the following properties hold for all edges $\ve, \ve' \in \ce$.

\noindent \textsc{Transitive:} If $t(\ve) \cap s(\ve')$ is non-empty, then $t(\ve) \cap s(\ve') = \{K\}$ for some $K \in \cv$ and
\begin{align}
\label{glued}
(s(\ve) \cup (s(\ve') - \{K\}), (t(\ve) - \{K\}) \cup t(\ve')) \in \ce
\end{align}
\hspace{21.5mm} where each of the above unions is of disjoint sets.

\noindent \textsc{Acyclic:} \ \ \ \ If $t(\ve) \cap s(\ve')$ and $t(\ve') \cap s(\ve)$ are both non-empty, then $\ve = \ve' = (\{K\}, \{K\})$ for

\hspace{17.5mm} some $K \in \cv$.
\end{definition}

An edge of the form $(\{K\}, \{K\})$ is called a {\em loop}.  Transitivity ensures that adjacent edges $\ve$ and $\ve'$ may be glued to form a single edge \eqref{glued} in a well-defined manner.  Acyclicity ensures that the only edges with overlapping source and target are loops, and that loops cannot result from gluing distinct adjacent edges.

\begin{figure}[htp]
\centering
\includegraphics[width=156mm]{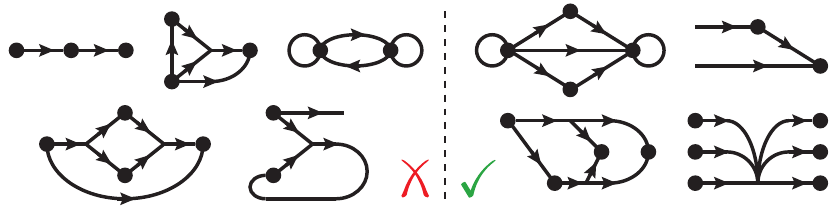}
\caption{Of the nine directed hypergraphs above, only the four at right are higraphs.  All vertices are shown as black dots.  The bottom-right higraph has only one edge; the source and target of this edge each consist of three vertices.  Two edges in the top-right higraph have empty source; gluing the upper two edges results in the lower edge.}
\label{fig:higraphRules}
\end{figure}

\subsection{Paths}

We recall the notion of a directed tree.

\begin{definition}
A {\em directed tree} $T$ is a digraph which is acyclic as an undirected graph.  For clarity, we will refer to the vertices and edges of $T$ as {\em nodes} and {\em arrows}, respectively.
\end{definition}

A path in a higraph has the structure of a directed tree -- the nodes index the edges of the path, while the arrows index the breaks between edges.

\begin{definition}
\label{def:path}
A {\em path} $\rho = (T, \rho)$ in a higraph $\ch = (\cv, \ce)$ consists of a non-empty, directed tree $T = (V, E)$ together with a map $\rho: V \to \ce$ such that:
\begin{enumerate}
\item The map $\rho$ induces a well-defined map $\rho: E \to \cv$ such that $$\{\rho(e)\} = t(\rho(s(e))) \cap s(\rho(t(e))).$$
\item If distinct arrows $e$ and $e'$ have the same source or the same target, then $\rho(e) \neq \rho(e')$.
\end{enumerate}
Paths $\rho$ and $\rho'$ are {\em isomorphic} if there is a digraph isomorphism $f: T \to T'$ of the associated directed trees such that $\rho = \rho' f$.
\end{definition}

In fact, the automorphism group of a path is trivial.  This is clear if $\rho$ covers a loop.  If $\rho$ does not cover a loop, then transitivity and acyclicity imply that the leaves of $T$ are mapped to distinct edges.  So in this case, any automorphism of $\rho$ must fix the leaves of $T$, and is therefore the identity.  Thus, paths are unique up to unique isomorphism, and we consider isomorphic paths to be the same.

\begin{definition}
Let $\cp$ denote the set of paths in $\ch$.  Given a path $\rho$, the {\em length} of $\rho$ is the number of edges it contains, given by $l(\rho) = |V|.$  We refer to the arrows of $T$ as {\em breaks}. The number of breaks in $\rho$ is given by $b(\rho) = |E| = l(\rho) - 1.$
The {\em source} of a path $\rho$ is the subset of $\cv$ given by
$$s(\rho) = \bigcup_{v \in V} \left( s(\rho(v)) - \{\rho(e) \, | \, t(e) = v\} \right).$$
The {\em target} of a path $\rho$ is the subset of $\cv$ given by
$$t(\rho) = \bigcup_{v \in V} \left( t(\rho(v)) - \{\rho(e) \, | \, s(e) = v\} \right).$$
Transitivity implies that the union defining $s(\rho)$ is of disjoint sets, and similarly for $t(\rho)$.  Thus each vertex $K \in s(\rho)$ is {\em sourced} by a unique node $v \in V$.  Similarly, each vertex $K \in t(\rho)$ is {\em targeted} by a unique node $v \in V$.
\end{definition}

\subsection{Path operations}

We next define several ways to combine or modify paths.  {\em Concatenation} joins two compatible paths together.  {\em Swapping} cuts out a subpath of a path and replaces it with another compatible path.  {\em Refinement} and {\em gluing} are special cases of swapping.  Refinement expands a single edge to a compatible path, while gluing contracts a subpath to a single edge.

\subsubsection{Concatenation}

\begin{definition}
The partially-defined {\em concatenation} map $c: \cp \times \cp \dashrightarrow \cp$ joins compatible paths together.  Consider paths $\rho_1$ and $\rho_2$ with associated trees $T_1 = (V_1, E_1)$ and $T_2 = (V_2, E_2)$.  The path $\rho = c(\rho_1,\rho_2)$ is defined precisely when $t(\rho_1) \cap s(\rho_2)$ is non-empty, in which case the intersection contains a single vertex $K \in \cv$ targeted by $v_1 \in V_1$ and sourced by $v_2 \in V_2$.  The directed tree $T = (V, E)$ associated to $\rho$ is given by $$V = V_1 \cup V_2$$ and $$E = E_1 \cup E_2 \cup \{(v_1, v_2)\}.$$  The map $\rho : V \to \ce$ is given by $\rho_1 \sqcup \rho_2$.  The induced map $\rho : E \to \cv$ sends $(v_1,v_2)$ to $K$.
\end{definition}

\subsubsection{Swapping}

Let $\mathcal{Q} = \{(\rho, T') \, | \, \rho = (T, \rho) \in \cp, \ T' \text { is a non-empty subtree of } T\}.$

\begin{definition}
The partially-defined {\em swapping map} $w: \mathcal{Q} \times \cp \dashrightarrow \cp$ replaces a subpath of a path with a compatible path.  Given a path $\rho_1 = (T_1, \rho_1)$ and non-empty subtree $T' = (V', E')$ of $T_1$, we obtain a subpath $\rho' = (T', \rho')$ by setting $\rho' = \rho_1|_{V'}$.  The path $\rho_2 = w((\rho_1, T'), \rho'')$ is defined precisely when $s(\rho') = s(\rho'')$ and $t(\rho') = t(\rho'')$.  In this case, the associated tree $T_2 = (V_2, E_2)$ has node set given by
$$V_2 = (V_1 - V') \cup V''.$$
To describe the arrow set $E_2$, note that for each arrow $e \in (V_1 - V') \times V'$, the vertex $\rho'(t(e)) \in s(\rho'')$ is sourced by a unique node $t''(e)$ of $T''$.  Similarly, for each arrow $e \in V' \times (V_1 - V')$, the vertex $\rho'(s(e)) \in t(\rho'')$ is targeted by a unique node $s''(e)$ of $T''$.  We set
$$E_2 = \left(E_1 \cap \left((V_1 - V') \times (V_1 - V')\right)\right) \cup E_s \cup E_t \cup E'' \subset V_2 \times V_2$$
where
$$E_s = \{(s(e),t''(e)) \, | \, e \in (V_1 - V') \times V' \}$$
and
$$E_t = \{(s''(e),t(e)) \, | \, e \in V' \times (V_1 - V') \}.$$
The map $\rho : V_2 \to \ce$ is induced by $\rho_1$ on $V_1 - V'$ and by $\rho''$ on $V''$.  Note that $\rho_2$ has the same source and target as the original path $\rho_1$.
\end{definition}

Since swapping preserves source and target, swapping and concatenation commute in the following sense.  Fix paths $\rho_1$, $\rho_2$, and $\rho''$, and let $T'_1$ and $T'_2$ be non-empty subtrees of $T_1$ and $T_2$, respectively.  Then 
\begin{align}
\label{eqn:rc1}
w((c(\rho_1,\rho_2), T'_1), \rho'') &= c(w((\rho_1, T'_1), \rho''), \rho_2) \\
\label{eqn:rc2}
w((c(\rho_1,\rho_2), T'_2), \rho'') &= c(\rho_1, w((\rho_2, T'_2), \rho''))
\end{align}
where the left-hand side of each equation is defined if and only if the right-hand side is defined.

Of course, if we swap out a subpath and then swap it back in, we end up where we started.  More precisely, if we define the map $W : \mathcal{Q} \times \mathcal{P} \to \mathcal{Q} \times \mathcal{P}$ by
\begin{align}
\label{Winverse}
W((\rho_1, T'), \rho'') = ((w((\rho_1,T'), \rho''), T''), \rho'),
\end{align}
then $W \circ W$ is the identity map.  We are most interested in the following two special cases of swapping, which are mutual inverses in the sense of \eqref{Winverse}.

\subsubsection{Refinement}

We first suppose that the subpath $\rho'$ consists of a single edge $\ve^* = \rho'(v^*)$.  Let $$\mathcal{R} = \{(\rho, v^*) \, | \, \rho = (T, \rho) \in \cp, \ v^* \text { is a vertex of } T\}.$$

\begin{definition}
The partially-defined {\em refinement map} $r : \mathcal{R} \times \mathcal{P} \dashrightarrow \mathcal{P}$ replaces a single edge $\ve^* = \rho_1(v^*)$ of a path $\rho_1$ with a compatible path $\rho''$:
$$r((\rho_1,v^*), \rho'') = w((\rho_1, v^*), \rho'').$$
\end{definition}

\subsubsection{Gluing}  We next suppose that the substituted path $\rho''$ consists of a single edge $\ve^*$.
\begin{definition}
The {\em gluing map} $q: \mathcal{Q} \to \mathcal{P}$ replaces a subpath $\rho'$ of a path $\rho_1$ with the single edge $\ve^* = (s(\rho'), t(\rho'))$:  
$$q(\rho_1,T') = w((\rho_1, T'), \ve^*).$$
Note that $\ve^*$ is indeed an edge of $\ch$ by transitivity.
\end{definition}

By definition, the tree $T_2$ underlying $\rho_2 = q(\rho_1,T')$ results from contracting the subtree $T'$ of $T_1$ to a single vertex $v^*$, with $\rho_2(v^*) = \ve^*$.  In particular, the projection map $\pi: V_1 \to V_2$ given by
\begin{align}
\label{pi}
\pi(v) =
\left\{\begin{array}{ll}
v &\text{if } v \in V_1 - V', \\
v^* &\text{if } v \in V'.
\end{array}\right.
\end{align}
induces a bijection $\pi: (E_1 - E') \to E_2$ between the arrows of $T_1$ that are outside $T'$ and the arrows of $T_2$.

\begin{definition}
Let $e$ be a break in a path $\rho$.  We use the notation $q(\rho, e)$ to denote the result of gluing the path $\rho$ along the one-arrow subtree induced by $e$.  Note that gluing a path at a break reduces length by 1.
\end{definition}

\begin{definition}
A {\em k-step gluing sequence} on a path $\rho_1$ is a $k$-tuple $(e_1, \dots, e_k)$
with $e_i$ a break on $\rho_i$ with $\rho_{i+1} = q(\rho_i, e_i)$ for $1 \leq i \leq k$.  We may write this as
$$\rho_1 \xrightarrow{e_1} \rho_2 \xrightarrow{e_2} \cdots \xrightarrow{e_k} \rho_{k+1}$$
to emphasize the paths along the way.  An isomorphism of $k$-step gluing sequences is a sequence of isomorphisms between the corresponding paths which identify the corresponding breaks.
\end{definition}

Gluing sequences are unique up to unique isomorphism, so we consider isomorphic gluing sequences to be the same.
The following lemma is established in Section \ref{sec:glueseqhi}.

\begin{lemma}
\label{lem:factorial1}
The number of $k$-step gluing sequences between any two fixed paths in a higraph is a multiple of $k!$.
\end{lemma}

\subsection{Path DGA}

The path algebra on a digraph is associative but not commutative.  The following generalization to higraphs is neither associative nor commutative in general.

\begin{definition}
The {\em path DGA} of the higraph $\ch$ over the field $\mathbb{F}$ consists of a graded algebra $\ca$ equipped with a differential $\delta : \ca \to \ca$.  The algebra $\ca$ has underlying vector space $\mathbb{F}\langle\cp\rangle$, with the subspace $\ca^k$ in grading $k$ spanned by the set $\cp^k$ of paths of length $k$.  Multiplication is bilinear and given on paths by $\rho_1\rho_2 = c(\rho_1, \rho_2)$ if defined, and $\rho_1\rho_2  = 0$ otherwise.

The differential $\delta$ is determined by the Leibniz rule and linearity once we specify its value on each edge $\ve \in \ce$ regarded as an element of $\ca$.  The value of $\delta$ on an edge $\ve$ equals the sum of all refinements of $\ve$ to a path of length 2.  In particular, the differential has degree $1$.
\end{definition}

Since a path $\rho$ may often be expressed as a fully-associated product of edges in several ways, we should take care that $\delta$ is well-defined.  This issue is resolved by noting that we may define $\delta$ directly on a path $\rho = ((V,E), \rho)$ by
\begin{align}
\label{eqn:delta}
\delta \rho = \sum_{v \in V} \sum_{\rho' \in \cp^2} r((\rho, v), \rho')
\end{align}
where $r((\rho, v), \rho')$ is defined to be zero on invalid input.  This definition is consistent with our earlier definition of $\delta \ve$, and the Leibniz rule now follows from \eqref{eqn:rc1} and \eqref{eqn:rc2}:
\begin{align*}
\delta (\rho_1\rho_2) &= \sum_{v \in V} \sum_{\rho' \in \cp^2} r((\rho_1\rho_2, v), \rho') \\
&= \sum_{v \in V_1} \sum_{\rho' \in \cp^2} r((\rho_1\rho_2, v), \rho') +  \sum_{v \in V_2} \sum_{\rho' \in \cp^2} r((\rho_1\rho_2, v), \rho') \\
&= \left(\sum_{v \in V_1} \sum_{\rho' \in \cp^2} r((\rho_1, v), \rho')\right) \rho_2 + \rho_1\left(\sum_{v \in V_2} \sum_{\rho' \in \cp^2} r(\rho_2, v), \rho') \right) \\
&= (\delta\rho_1)\rho_2 +  \rho_1(\delta\rho_2).
\end{align*}

To see that $\delta$ is a differential, it is useful to introduce its adjoint $\partial$ with respect to the symmetric, bilinear pairing $$\langle \, \cdot \, , \, \cdot \, \rangle : \ca \times \ca \to \mathbb{F}$$ with $\langle \rho_1, \rho_2\rangle = 1$ if and only if $\rho_1 = \rho_2$.

\begin{definition}
\label{def:partial}
The map $\partial: \ca \to \ca$ is defined on the path $\rho$ with underlying tree $T=(V,E)$ by
\begin{align*}
\partial \rho = \sum_{e \in E} q(\rho, e).
\end{align*}
\end{definition}

\begin{proposition}
We have $$\langle \partial a, b \rangle = \langle a, \delta b \rangle$$
for all $a, b \in \ca$.  Furthermore, $\partial^2 = 0$ and $\delta^2 = 0$.
\end{proposition}

\begin{proof}
For the first statement, by linearity it suffices to verify that $$\langle \partial \rho_1, \rho_2 \rangle = \langle \rho_1, \delta\rho_2 \rangle.$$
The left-hand side counts the number of $1$-step gluing sequences $\rho_1 \xrightarrow{e_1} \rho_2$.  The right-hand side counts all ways to refine an edge of $\rho_2$ to a path of length $2$ to obtain $\rho_1$.  Equality follows from the fact that gluing and refinement are inverses as in \eqref{Winverse}.

To prove $\partial^2 = 0$, it suffices to show that $\langle \partial^2\rho_1, \rho_3\rangle$ vanishes for all paths $\rho_1$ and $\rho_3$.  Indeed, the term $\langle \partial^2\rho_1, \rho_3\rangle$ counts the number of 2-step gluing sequences $\rho_1 \xrightarrow{e_1} \rho_2 \xrightarrow{e_2} \rho_3$, which is even by Lemma \ref{lem:factorial1}.

Now $\delta^2 = 0$ because
$\langle \delta^2 a, b \rangle = \langle \delta a, \partial b \rangle = \langle a, \partial^2 b \rangle = \langle a, 0\rangle = 0.$
\end{proof}

\begin{definition}
We define the element $D \in \ca$ by $$D = \sum_{\ve \in \ce} \ve = \sum_{\rho \in \cp^1} \rho.$$
\end{definition}

\begin{proposition}  The element $D$ satisfies the structure equation
$$\delta D = D \circ D.$$
\end{proposition}
\begin{proof}
We claim $$\delta D = \sum_{\rho \in \cp^2} \rho = D \circ D.$$  The left-hand equality holds because each path $\rho$ of length 2 arises as the refinement of a unique edge $\ve = (s(\rho), t(\rho))$.  The right-hand equality holds because each path of length 2 arises as the product of two (not necessarily distinct) edges in a unique way.
\end{proof}

\section{Abstract paths}
\label{sec:abstractpaths}

In Section \ref{sec:bigraphs}, we will define the notion of the bigraph $\tch$ associated to a higraph $\ch$.  We will then systematically extend the arguments and results of Section \ref{sec:higraphs} to bigraphs.  In this section, we assemble the necessary lemmas, which concern the abstract form of a path in a bigraph.  We call such a form an abstract path, as it is defined without reference to a particular bigraph.  We now make this notion precise.

\begin{definition}
\label{def:abstractedge}
An {\em abstract vertex} is an element of the set $\cv_{ab} = \{o,u,s\}$.  The symbols $o$, $u$, and $s$ stand for {\em interior}, {\em unstable}, and {\em stable}, respectively.  Both unstable and stable abstract vertices are considered {\em boundary}.

An {\em abstract edge} is an element $$\ve = (\mu, (o_1,u_1,s_1), (o_2,u_2,s_2))$$ of the subset $$\mathcal{E}_{ab} \subset \mathbb{F} \times (\mathbb{N} \times \mathbb{N} \times \mathbb{N}) \times (\mathbb{N} \times \mathbb{N} \times \mathbb{N}).$$  This subset is cut out by the following two conditions: \\

\indent Condition 0: If $\mu = 0$ then $s_1 = u_2 = 0$.\\
\indent Condition 1: If $\mu = 1$ then $o_1 = o_2 = 0$. \\

\noindent An abstract edge $\ve$ is {\em interior} if $\mu = 0$ and {\em boundary} if $\mu = 1$.

\end{definition}

We think of the abstract edge $\ve$ as having $o_1 + u_1 + s_1$ incoming ends (or source vertices), of which $o_1$ are interior, $u_1$ are unstable, and $s_1$ are stable.  Similarly, $\ve$ has $o_2 + u_2 + s_2$ outgoing ends, of which $o_2$ are interior, $u_2$ are unstable, and $s_2$ are stable.  Condition $0$ says that the incoming ends of an interior abstract edge are interior or unstable, while the outgoing ends are interior or stable.  Condition $1$ says that the ends of a boundary abstract edge are boundary abstract vertices.

We can assemble abstract edges into an abstract path, so long as there are sufficiently many ends for concatenation. 

\begin{definition}
An {\em abstract path} is a triple $\tau = (T, \tau, \sigma)$ consisting of a non-empty directed tree $T$ and maps
\begin{align*}
\tau &: V \to \mathcal{E}_{ab} \\
\sigma &: E \to \cv_{ab}
\end{align*}
such that the following inequalities hold for each $v \in V$, with $\tau(v) = (\mu, (o_1,u_1,s_1), (o_2,u_2,s_2))$:
\begin{align*}
o_1 &\geq |\{e \in E \, | \, t(e) = v, \, \sigma(e) = o\}|, \\
u_1 &\geq |\{e \in E \, | \, t(e) = v, \, \sigma(e) = u\}|, \\
s_1 &\geq |\{e \in E \, | \, t(e) = v, \, \sigma(e) = s\}|, \\
o_2 &\geq |\{e \in E \, | \, s(e) = v, \, \sigma(e) = o\}|, \\
u_2 &\geq |\{e \in E \, | \, s(e) = v, \, \sigma(e) = u\}|, \\
s_2 &\geq |\{e \in E \, | \, s(e) = v, \, \sigma(e) = s\}|.
\end{align*}
An abstract path $\tau$ is {\em boundary} if all of its abstract edges are boundary.
\end{definition}

We interpret an arrow $e$ from node $v_1$ to node $v_2$ to signify that the abstract edges $\tau(v_1)$ and $\tau(v_2)$ are concatenated at an abstract vertex of type $\sigma(e)$.  In this sense, an abstract path has a well-defined number of source and target abstract vertices of each type.  For example, the number of incoming stable ends is given by
$$s_1(\tau) = \left(\sum_{v \in V} s_1(v)\right) - |\sigma^{-1}(s)|,$$
while the number of outgoing unstable ends is given by
$$u_2(\tau) = \left(\sum_{v \in V} u_2(v)\right) - |\sigma^{-1}(u)|.$$

\begin{definition}
An abstract path $\tau$ is {\em legal} if $s_1(\tau) = u_2(\tau) = 0$; equivalently, $\tau$ is legal if its ends are compatible with an interior abstract edge (see Condition $0$ of Definition \ref{def:abstractedge}).
\end{definition}

It is convenient to pull back the terminology in Definition \ref{def:abstractedge}, so that an arrow $e \in E$ is interior, unstable, or stable according to the value of $\sigma(e)$.  Similarly, a node $v \in V$ is interior or boundary according to the value of $\mu(v)$.  We may decompose the node set $V$ as $V^0 \cup V^1$ where $V^0 = \mu^{-1}(0)$ denotes the set of interior nodes and $V^1 = \mu^{-1}(1)$ denotes the set of boundary nodes.

\begin{definition}  The {\em interior neighborhood} of a node $v \in V$ is the set of adjacent interior nodes:
$$\text{Nbhd}^o(v) = V^o \cap \{v' \in V \, | \, (v, v') \in E \text{ or } (v', v) \in E \}.$$
The {\em interior star} of $v$ is the subtree $\text{Star}^o(v)$ of $T$ induced by $\{v\} \cup \text{Nbhd}^o(v)$.
\end{definition}

\subsection{Gluing moves on abstract paths}
We have introduced all the necessary terminology to define the gluing moves on abstract paths.  We illustrate these moves in Figure \ref{fig:gluerule}, while precise definitions are given below.

\begin{figure}[htp]
\centering
\includegraphics[width=150mm]{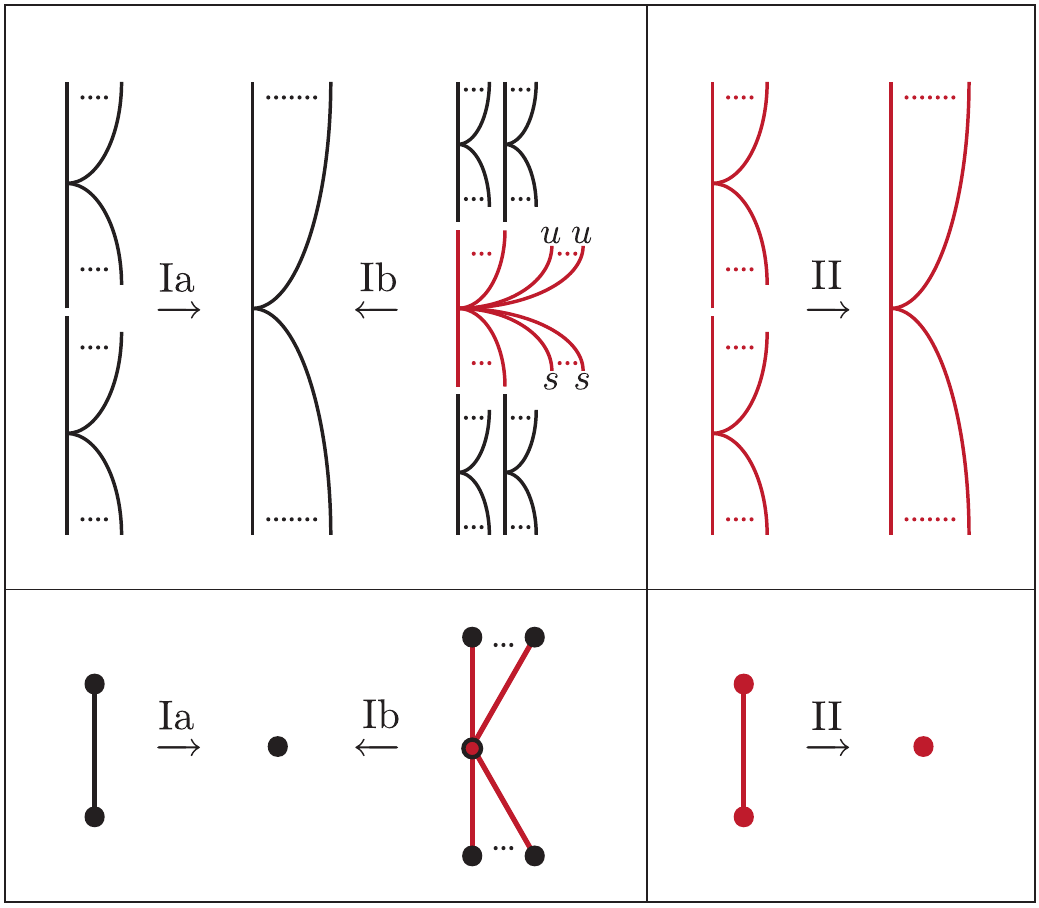}
\caption{The three gluing moves are Ia (interior), Ib (boundary-obstructed), and II (boundary).  Move Ia glues two adjacent interior edges to obtain a single interior edge.  Move Ib is applicable in the presence of a boundary edge such that the subpath obtained by appending adjacent interior edges is legal.  Move Ib replaces this subpath with a single interior edge.  Move II glues two adjacent boundary edges to obtain a single boundary edge.  The effect of each gluing move is to contract the underlying tree, as shown in the second row.  Interior (resp., boundary) breaks and abstract edges above correspond to black (resp., red) arrows and nodes below.  A red node with a black border is valid for Move Ib (a red node without a black border may or may not be valid).}
\label{fig:gluerule}
\end{figure}

\begin{definition}
The three gluing moves on an abstract path $\tau = (T, \tau, \sigma)$ with $T = (V, E)$ are: \\

\noindent Ia. \textsc{Interior:} Given an arrow $e \in E$ between interior nodes,
contract $e$ and set $$\tau(v^*) = \tau(s(e)) + \tau(t(e)) - (0, (1,0,0), (1,0,0)).$$

\noindent Ib. \textsc{Boundary-obstructed:} Given a boundary node $v \in V$ with
\begin{align*}
o_1(v) + s_1(v) &=  |\{e \in E \, | \, t(e) = v\}|, \\
o_2(v) + u_2(v) &=  |\{e \in E \, | \, s(e) = v\}|,
\end{align*}
\hspace{5.1mm} contract $\text{Star}^o(v)$ and set $$\tau(v^*) = \sum_{v' \in \text{Nbhd}^o(v)} \tau(v') - \tau(v) + (1,(0,0,0),(0,0,0)).$$

\noindent II. \textsc{Boundary:} Given an arrow $e \in E$ between boundary nodes, contract $e$ and set
$$\tau(v^*) = \left\{
\begin{array}{cl}
\tau(s(e)) + \tau(t(e)) - (1, (0,1,0), (0,1,0)) &\text{if } \sigma(e) = u, \\
\tau(s(e)) + \tau(t(e)) - (1, (0,0,1), (0,0,1)) &\text{if } \sigma(e) = s.
\end{array}\right.$$
We denote the resulting abstract paths by $q^{Ia}(\tau, e)$, $q^{Ib}(\tau, v)$, and $q^{II}(\tau, e)$, respectively.
\end{definition}

\begin{definition}
\label{def:marker}
A {\em marker} on $\tau$ is a node or arrow of the underlying tree $T$.  We will use $X$ as shorthand for the marker set $V \cup E$.  Note that at most one of $q^{Ia}(\tau, x)$, $q^{Ib}(\tau, x)$, and $q^{II}(\tau, x)$ is defined for each $x \in X$.  If one is indeed defined, we denote the resulting glued abstract path by $\tilde q(\tau, x)$.
\end{definition}

\subsection{Abstract paths of low weight}
We now define a notion of weight for abstract paths.  Since all vertices, edges, and paths in this section are abstract, we will drop the descriptor.

\begin{definition}
\label{def:weight}
The {\em weight} of an edge $\ve = (\mu, (o_1,u_1,s_1), (o_2,u_2,s_2))$ is given by:
$$w(\ve) = \left\{
\begin{array}{cl}
1 &\text{if } \mu = 0, \\
2 - s_1 - u_2 &\text{if } \mu = 1.
\end{array}\right.$$
The {\em weight} of a path is the sum of the weights of its edges:
$$w(\tau) = \sum_{v \in V} w(\tau(v)).$$
\end{definition}

\begin{lemma}
\label{lem:glueweight}
Each gluing move reduces weight by 1.
\end{lemma}
\begin{proof}
For Move Ia, the result is immediate.

For Move Ib, it suffices to show that the subpath $\tau'$ corresponding to $\text{Star}^o(v)$ has weight 2, since this subpath is being replaced by an edge of weight $1$.  We induct on the length $l(\tau')$.  If $l(\tau') = 1$, then $\tau'$ consists of a single legal boundary edge $\ve = \tau(v)$, which has weight 2 by definition.  If $l(\tau') > 1$, then $\tau'$ includes an interior edge $\ve'$.  We suppose that $\ve'$ precedes $\ve = \tau(v)$.  Then the break $e$ between $\ve'$ and $\ve$ is stable.  Let $\tau''$ be the subpath of $\tau'$ that results from pruning $\ve'$ from $\tau'$.  Then $$w(\tau') = w(\tau'') + 1.$$  The subpath $\tau''$ is illegal due to a single incoming stable end.  Converting this stable end to an unstable end yields a legal path $\tau'''$ with $$w(\tau''') = w(\tau'') + 1.$$
Furthermore, $\tau'''$ is a candidate for Move II of length $l(\tau')-1$, so by induction we have $$w(\tau''') = 2.$$ Together these equations imply $w(\tau') = 2$.  The proof is similar if $\ve'$ follows $\ve$.

For Move II, let $\ve_1 = \tau(s(e))$ and $\ve_2 = \tau(t(e))$ be the consecutive boundary edges and let $\ve$ be the resulting glued boundary edge.  If $e$ is stable, then
\begin{align*}
w(\ve_1) + w(\ve_2) &= (2 - s_1(\ve_1) - u_2(\ve_1)) + (2 - s_1(\ve_2) - u_2(\ve_2)) \\
&= 1+ 2 - (s_1(\ve_1) + s_1(\ve_2) - 1) - (u_2(\ve_1) + u_2(\ve_2)) \\
&= 1 + 2 + s_1(\ve) - u_2(\ve) \\
&= 1 + w(\ve).
\end{align*}
The proof is similar if $e$ is unstable.
\end{proof}

In order to classify paths of low weight, we introduce several counting functions.
\begin{definition}
Given a path $\tau$, let $O(\tau)$, $U(\tau)$, and $S(\tau)$ denote the number of breaks of each type:
\begin{align*}
O(\tau) &= |\sigma^{-1}(o)| \\
U(\tau) &= |\sigma^{-1}(u)| \\
S(\tau) &= |\sigma^{-1}(s)|
\end{align*}
Let $B(\tau)$ denote the number of boundary edges:
$$B(\tau) = |\mu^{-1}(1)|$$
\end{definition}

\begin{proposition}
\label{prop:legalweight}
For any legal path $\tau$, we have
\begin{align}
\label{legalweight}
w(\tau) = 1 + O(\tau) + B(\tau).
\end{align}
In particular, a legal path has weight $1$ if and only if it consists of a single interior edge.
\end{proposition}

\begin{proof}
We induct on the length $l(\tau)$.  The base case $l(\tau) = 1$ follows from the fact that a legal boundary edge has weight 2.

We now suppose that $l(\tau) > 1$.  Note that Move Ib, applied to a legal boundary edge, reduces both $w$ and $B$ by 1 while preserving $O$ (and $l$), and therefore preserves \eqref{legalweight}.  So we may assume without loss of generality that $\tau$ has no legal boundary edges.  Now $\tau$ falls into at least one of the following cases:

{\em Case 1:} $\tau$ has two adjacent interior edges.

{\em Case 2:} $\tau$ has a boundary edge without adjacent boundary edges.

{\em Case 3:} $\tau$ has two adjacent boundary edges.

In Case 1, we can apply Move Ia.  This reduces $l$ and $O$ by 1 and preserves $B$.  In Case 2, since $\tau$ is legal, we can apply Move Ib to the interior star of this (illegal) boundary edge to reduce $l$ by at least 1.  This preserves $O$ and reduces $B$ by 1.  In Case 3, we can apply Move II.  This reduces $l$ and $B$ by 1 and preserves $O$.  So in all three cases, the right-hand side of \eqref{legalweight} is reduced by 1.  By Lemma \ref{lem:glueweight}, the left-hand side of \eqref{legalweight} is reduced by 1 as well.  Thus \eqref{legalweight} is preserved.  Since length is reduced in all cases, we are done by induction.

We now turn to the final claim.  If $\tau$ is legal and $w(\tau) = 1$, then \eqref{legalweight} implies that $\tau$ has no boundary edges and no interior breaks.  Since breaks between adjacent interior edges are necessarily interior, we conclude that $\tau$ consists of a single interior edge.
\end{proof}

\begin{corollary} [Classification of legal paths of low weight]
\label{cor:legalweight2}
\ \\
(i) A legal path has weight 2 if and only if it has the form Ia or Ib. \\
(ii) A boundary path has length 2 if and only if it has the form II.
\end{corollary}
\begin{proof}
For (i), suppose $\tau$ is legal and $w(\tau) = 2$.  By \eqref{legalweight}, there are two possibilities:

{\em Case 1:} $O(\tau) = 1$ and $B(\tau) = 0$.

{\em Case 2:} $O(\tau) = 0$ and $B(\tau) = 1$.

In Case 1, $\tau$ consists of two interior edges, giving the form Ia.

In Case 2, no two interior edges may be adjacent, so $\tau$ consists of a single boundary edge $\tau(v)$ and some number of adjacent interior edges.  Thus $\tau$ coincides with Star$^o(v)$ .  Since $\tau$ is legal, it has the form Ib.

Claim (ii) is tautological.
\end{proof}

\begin{corollary}
\label{cor:genweight}
For any path $\tau$, we have
\begin{align}
\label{genweight}
w(\tau) = 1 + O(\tau) + B(\tau) - s_1(\tau) - u_2(\tau).
\end{align}
\end{corollary}
\begin{proof}
We build a legal path $\tau'$ from $\tau$ by concatenating the abstract edge $(1, (0,1,0), (0,0,1))$ before each incoming stable end of $\tau$ and after each outgoing unstable end of $\tau$.  This adds $s_1(\tau) + u_2(\tau)$ boundary edges of weight $2$, so
\begin{align*}
w(\tau) &= w(\tau') - 2(s_1(\tau) + u_2(\tau)) \\
&= (1 + O(\tau') + B(\tau')) - 2(s_1(\tau) + u_2(\tau)) \\
&= 1 + O(\tau) + (B(\tau) + s_1(\tau) + u_2(\tau)) - 2(s_1(\tau) + u_2(\tau)) \\
&= 1 + O(\tau) + B(\tau) - s_1(\tau) - u_2(\tau)
\end{align*}
where the second equality uses Proposition \ref{prop:legalweight}.
\end{proof}

\subsubsection{Good tree-like paths}

\begin{definition}
A path is {\em tree-like} if every edge has exactly one outgoing end.  For such a path $\tau$, let $B_s(\tau)$ denote the number of boundary edges with stable outgoing end:
$$B_s(\tau) = |\{v \in V \, | \, \mu(v) = 1 \text{ and } s_2(v) = 1\}|.$$
\end{definition}

\begin{definition}
A path is {\em good} if each end is interior or unstable.
A break is {\em good} if it is interior or unstable.  Let $G(\tau)$ denote the number of good breaks in $\tau$:
$$G(\tau) = O(\tau) + U(\tau).$$
\end{definition}

We will use the following corollary of Corollary \ref{cor:genweight} to characterize good tree-like paths of low weight. 

\begin{corollary}
If $\tau$ is a legal tree-like path or a good tree-like path, then
\begin{align}
\label{goodweight}
w(\tau) = 1 + G(\tau) + B_s(\tau).
\end{align}
\end{corollary}
\begin{proof}
If $\tau$ ends in $o$ or $s$, then
\begin{align*}
w(\tau) &= 1 + O(\tau) + B(\tau) \\
&= 1 + O(\tau) + (U(\tau) + B_s(\tau)) \\
&= 1 + (O(\tau) + U(\tau)) + B_s(\tau) \\
&= 1 + G(\tau) + B_s(\tau).
\end{align*}
If $\tau$ ends in $u$, then
\begin{align*}
w(\tau) &= 1 + O(\tau) + B(\tau) - 1 \\
&= 1 + O(\tau) + (U(\tau) + 1 + B_s(\tau)) - 1 \\
&= 1 + (O(\tau) + U(\tau)) + B_s(\tau) \\
&= 1 + G(\tau) + B_s(\tau).
\end{align*}
\end{proof}

\begin{proposition}[Classification of good tree-like paths of low weight]
\label{prop:goodcases}
\ \\
(i) A good tree-like path has weight 1 if and only if it has a form in $(a)$ of Figure \ref{fig:treebreak}. \\
(ii) A good tree-like path has weight $2$ if and only if it has a form in $(c)$ or $(d)$ of Figure \ref{fig:treebreak}.
\end{proposition}
\begin{proof}
Let $\tau$ be a good tree-like path, with final edge $\ve$.

For (i), if $\tau$ has weight $1$ then by \eqref{goodweight} all breaks of $\tau$ must be stable and follow interior edges.  So if $\ve$ is interior, then no edges may precede $\ve$.  And if $\ve$ is boundary, then each incoming end of $\ve$ is preceded by at most one interior edge.  These two cases yield the two forms in $(a)$.

For (ii), If $\tau$ has weight $2$ then by \eqref{goodweight} all breaks but one are both stable and follow interior edges.  Let $e$ be this rogue break.  Note that if $\ve$ is interior than it is necessarily preceded by an interior or unstable break.  We are therefore left with six mutually exclusive cases:

{\em Case 1}: $e$ is interior and $\ve$ is interior.

{\em Case 2}: $e$ is interior and $\ve$ is boundary.

{\em Case 3}: $e$ is unstable and $\ve$ is interior.

{\em Case 4}: $e$ is unstable, $\ve$ is boundary, and $t(e)$ is interior.

{\em Case 5}: $e$ is unstable, $\ve$ is boundary, and $t(e)$ is boundary.

{\em Case 6}: $e$ is stable and $\ve$ is boundary. \\
Note that in Cases $3$ and $4$, the edge $\ve'$ preceding the unstable break $e$ is boundary and all edges adjacent to $\ve'$ are interior.  Since $\tau$ is good, the interior star of $\ve'$ is therefore legal.

In each case, we claim that $\tau$ must have a particular form.  The six forms are shown sequentially from left to right in the second row of Figure \ref{fig:treebreak}.  In each case, the line of argument is to perform a gluing move on $\tau$, apply (i) to force the form of the resulting good tree-like path of weight 1, and then reverse the gluing move to recover the form of $\tau$.  In Cases 1 and 2, we apply Move Ia at $e$.  In Cases 3 and 4, we apply Move Ib at the edge $\ve'$ preceding $e$.  In Cases 5 and 6, we apply Move II at $e$.  The six resulting paths are shown sequentially in the third row of Figure \ref{fig:treebreak} (the seventh path does not play a role in this argument).  Reversing the gluing move, we obtain the six forms in the second row.
\end{proof}

\begin{figure}[htp]
\centering
\includegraphics[width=130mm]{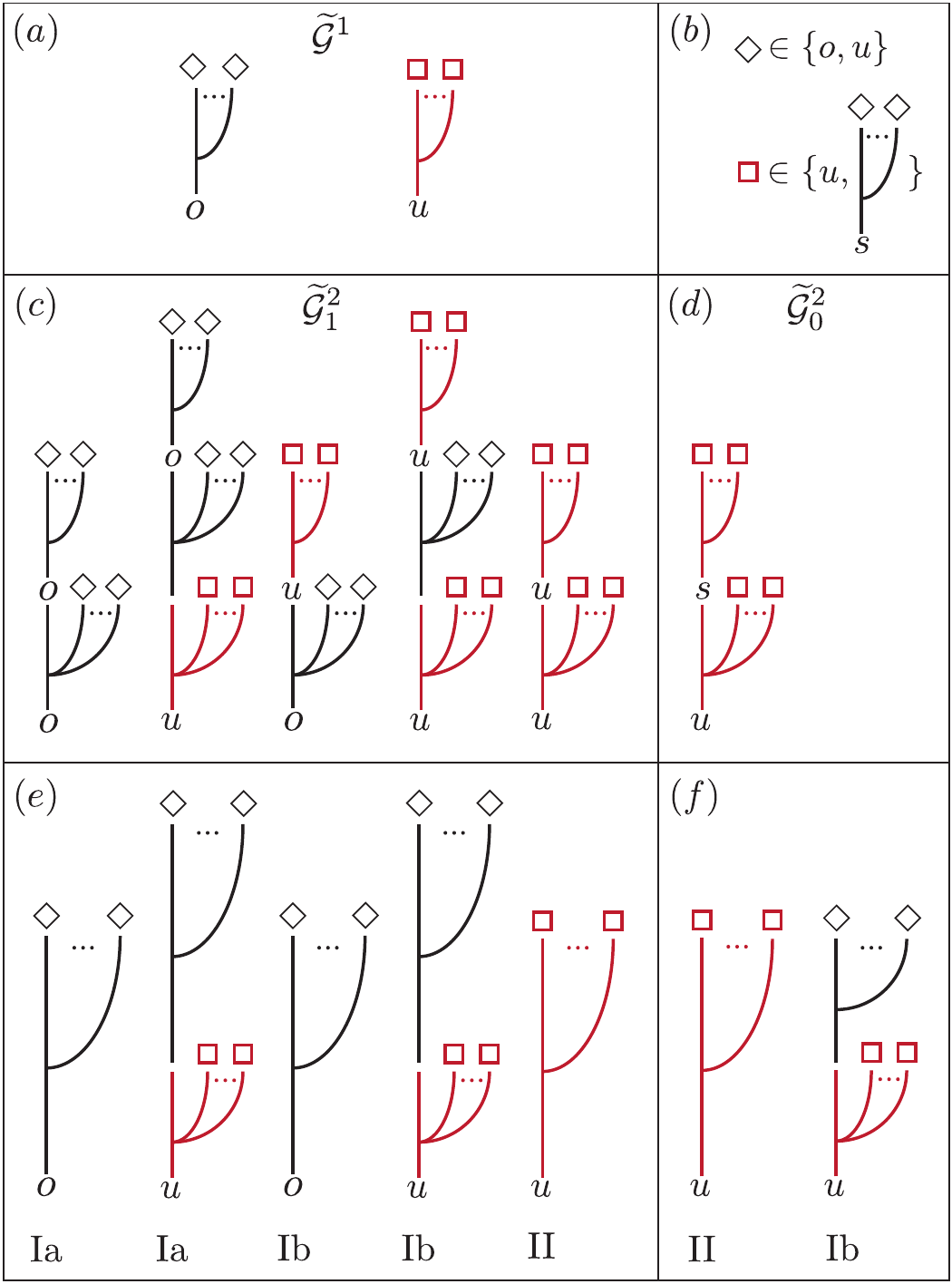}
\caption{The forms of a good tree-like weight 1 path are shown in $(a)$, where the diamond and square symbols are members of the sets in $(b)$ and ellipses indicate any finite number of incoming ends (including zero).  The forms of a good, tree-like weight 2 path are shown in $(c)$ and $(d)$.  Those with one good break are in $(c)$, each admitting a unique gluing shown directly below in $(e)$.  Those with no good break are shown in $(d)$, each admitting the two gluings shown in $(f)$.  The notation $\mathcal{\widetilde G}^k_m$ is introduced in Section \ref{sec:treelike}.}
\label{fig:treebreak}
\end{figure}

The following corollary of Proposition \ref{prop:goodcases} is crucial for generalizing the structure equation to bigraphs.

\begin{corollary}[The key to good breaks]
\label{cor:onetwoglue}
Let $\tau$ be a good tree-like path of weight $2$.  Then $\tau$ has at most one good break.  Furthermore: \\
(i) If $\tau$ has one good break, then $\tau$ can be glued in precisely one way. \\
(ii) If $\tau$ has no good breaks, then $\tau$ can be glued in precisely two ways.
\end{corollary}
\begin{proof}
The path $\tau$ has at most one good break by \eqref{goodweight}.  By Proposition \ref{prop:goodcases}, $\tau$ has one of the six forms shown in $(c)$ and $(d)$ of Figure \ref{fig:treebreak}.  Each form in $(c)$ has a good break and may be glued in precisely the one way shown directly below in $(e)$.  The form in $(d)$ has no good break and may be glued in precisely the two ways shown in $(f)$.
\end{proof}


\section{Bigraphs}
\label{sec:bigraphs}
We now introduce the bigraph $\tch$ associated to a higraph $\ch$.  The ``b'' stands for blow-up or boundary and ``bi'' signifies directedness.  The bigraph $\tch = (\tcv, \tce)$ covers the higraph $\ch = (\cv, \ce)$ in the sense that there is a $3$-to-$1$ map $\eta: \tcv \to \cv$ which induces a map on edges that preserves the size of source and target.  The three lifts of a vertex in $\cv$ are labeled by the three elements of the abstract vertex set $\cv_{ab}$, and are called interior, unstable, and stable.  The $2^{|I| + |J|+1}$ lifts of an edge $(I,J) \in \ce$ are precisely the compatible abstract edges, of which $2^{|I| + |J|}$ are interior and $2^{|I| + |J|}$ are boundary.  We make this precise in the following definitions, which are followed by examples in Figures \ref{fig:bigraph} and \ref{fig:pantsgraph}.

\begin{definition}
The {\em bigraph} $\tch = (\tcv, \tce)$ associated to the higraph $\ch$ consists of a vertex set
$$\tcv = \cv \times \{o,u,s\}$$
and an edge set
$$\tce \subset \mathbb{F} \times 2^{\tcv} \times 2^{\tcv}$$
with $(\mu, \tilde{I}, \tilde{J}) \in \tce$ if and only if $|\eta(\tilde{I})| = |\tilde{I}|$, $|\eta(\tilde{J})| = |\tilde{J}|$, $(\eta(\tilde{I}), \eta(\tilde{J})) \in \ce$ and $g(\tilde{I}, \tilde{J}, \mu) \in \ce_{ab}$.

Here $\eta : 2^{\tcv} \to 2^{\cv}$ is induced by the projection $\eta: \tcv \to \cv$ and
$$g: \mathbb{F} \times 2^{\tcv} \times 2^{\tcv} \to \mathbb{F} \times (\mathbb{N} \times \mathbb{N} \times \mathbb{N}) \times (\mathbb{N} \times \mathbb{N} \times \mathbb{N})$$
is defined by
$$f(\mu, \tilde{I}, \tilde{J}) = (\mu, (\tilde o(\tilde{I}),\tilde u(\tilde{I}),\tilde s(\tilde{I})),(\tilde o(\tilde{J}),\tilde u(\tilde{J}),\tilde s(\tilde{J})))$$
where $\tilde o(\tilde{I}) = |\{(K, o) \in \tilde{I}\}|$, $\tilde u(\tilde{I}) = |\{(K, u) \in \tilde{I}\}|$, and $\tilde s(\tilde{I}) = |\{(K, s) \in \tilde{I}\}|$.
\end{definition}

\begin{definition}
The  {\em boundary bigraph}  $\bch = (\bcv, \bce)$ associated to the higraph $\ch$ is the directed hypergraph given by the boundary vertices and boundary edges of $\tch$.  Equivalently, 
$\bch$ has the vertex set
$$\bcv = \cv \times \{u,s\}$$
and the edge set
$$\bce \subset 2^{\bcv} \times 2^{\bcv}$$
with $(\bar{I}, \bar{J}) \in \bce$ if and only if $ |\eta(\bar{I})| = |\bar{I}|$, $|\eta(\bar{J})| = |\bar{J}|$, and $(\eta(\bar{I}), \eta(\bar{J})) \in \ce$.
\end{definition}

\begin{figure}[htp]
\centering
\includegraphics[width=135mm]{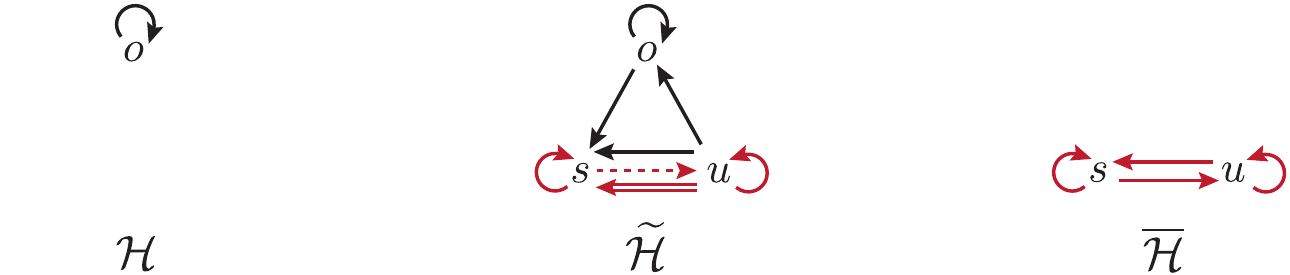}
\caption{The higraph $\ch$ above consists of a single vertex and a single loop.  The loop lifts to four interior edges and four boundary edges in the bigraph $\tch$.  The dotted edge has weight 0,  the doubled edge has weight 2, and all other edges have weight 1.  Only the boundary vertices and edges are present in the boundary bigraph $\bch$.}
\label{fig:bigraph}
\end{figure}

\begin{figure}[htp]
\centering
\includegraphics[width=155mm]{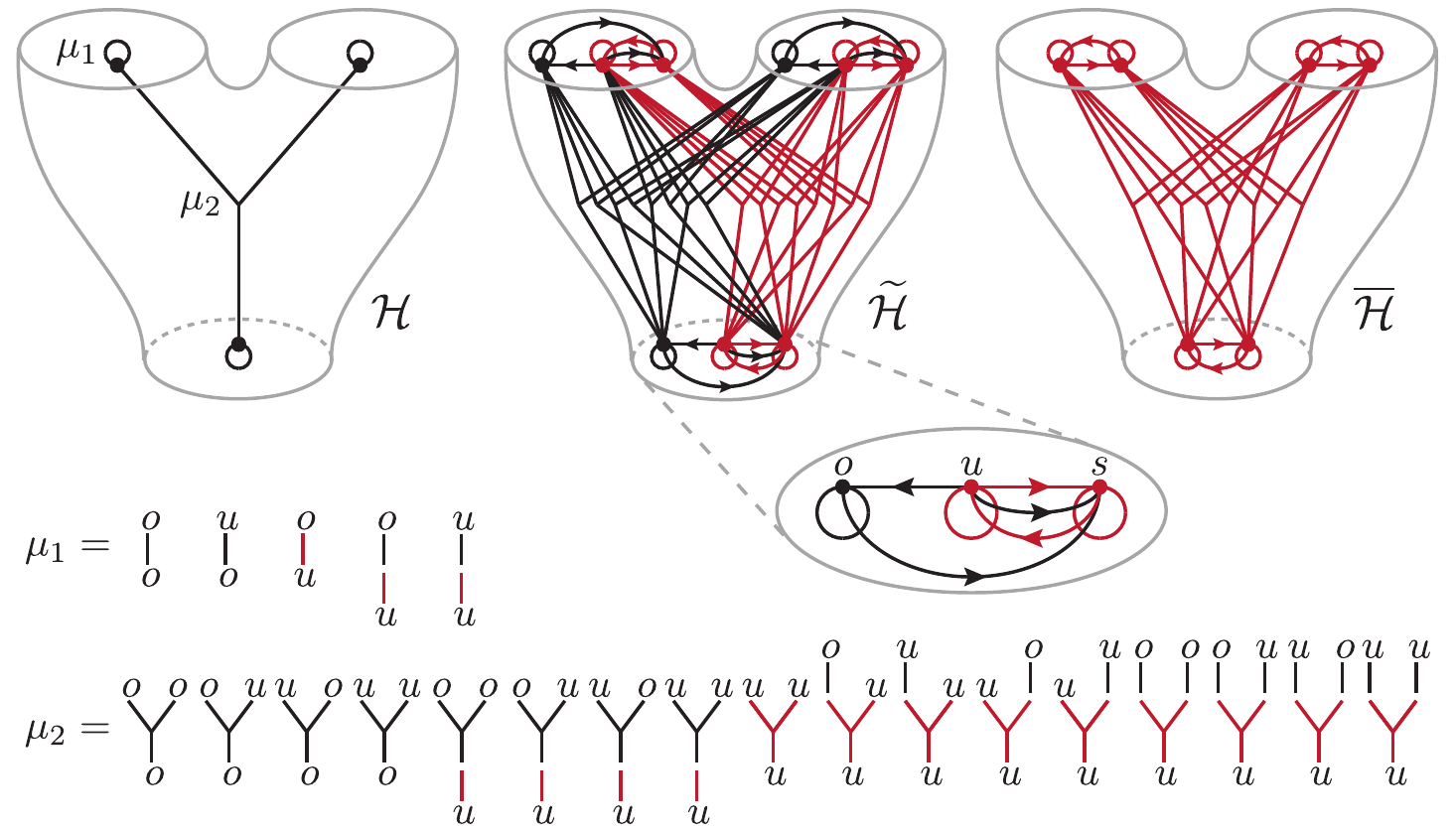}
\caption{The higraph $\ch$ above consists of three vertices and four edges.  Three edges are loops and the remaining edge $\ve$ has two vertices as source and one vertex as target.  The edge $\ve$ lifts to eight interior edges and eight boundary edges in the bigraph $\tch$.  Only the boundary vertices and edges are present in the boundary bigraph $\bch$.  Edge weights are not indicated.}
\label{fig:pantsgraph}
\end{figure}

\subsection{Paths}
Note that $\tch$ and $\bch$ are not higraphs because there are non-loop cycles in $\tch$ and $\bch$, though the edges of a cycle project to a common loop in $\ch$.  Furthermore, there are two lifts of each edge $(I,J) \in \ce$ to $\tce$ with the same source and target, namely $(I \times \{u\}, J \times \{s\},0)$ and $(I \times \{u\}, J \times \{s\},1)$.  Nevertheless, we may define paths just as before.

\begin{definition}
A {\em path} $\tr = (T, \tr)$ in a bigraph $\tch$ consists of a non-empty, directed tree $T = (V, E)$ together with a map $\tr: V \to \tce$ such that the following properties hold.
\begin{enumerate}
\item The map $\tr$ induces a well-defined map $\tr: E \to \tcv$ such that $$\{\tr(e)\} = t(\tr(s(e))) \cap s(\tr(t(e))).$$
\item If distinct arrows $e$ and $e'$ have the same source or the same target, then $\tr(e) \neq \tr(e')$.
\end{enumerate}
Paths $\rho$ and $\rho'$ are {\em isomorphic} if there is a digraph isomorphism $f: T \to T'$ of the associated directed trees such that $\tr = \tr' f$.
\end{definition}

As with higraphs, paths in a bigraph are unique up to unique isomorphism and we regard isomorphic paths as the same.  We may similarly define the notion of a path $\br$ in $\bch$.

Let $\tcp$ denote the set of paths in $\tch$, and let $\bcp$ denote the set of paths in $\bch$.  The operations of concatenation, gluing, and refinement carry over without change.  Since each path $\tilde{\rho}$ determines an abstract path $g(\tilde{\rho})$, we may pull back the three gluing moves from abstract paths to bigraph paths.

\begin{definition}
A {\em $k$-step gluing sequence} on a path $\tr_1$ is a $k$-tuple $$(x_1, x_2, \dots, x_k)$$ with $x_i$ a valid marker on $\tr_i$ and $\tr_{i+1} = \tilde q(\tr_i, x_i)$.  We may use the notation
$$\tr_1 \xrightarrow{x_1} \tr_2 \xrightarrow{x_2} \cdots \xrightarrow{x_k} \tr_{k+1}$$ to emphasize the paths along the way.
\end{definition}

As with the higraph case, gluing sequences in a bigraph are unique up to unique isomorphism.  Unlike the higraph case, the number of $k$-step gluing sequences between two fixed paths in a bigraph may not be a multiple of $k!$.  Indeed, Figure \ref{fig:eightPaths} gives an example with exactly eight $3$-step gluing sequences for the bigraph in Figure \ref{fig:bigraph}.  However, the analogue of Lemma \ref{lem:factorial1} does hold when $k = 2$.

\begin{lemma}
\label{lem:even}
The number of $2$-step gluing sequences between any two fixed paths in a bigraph is even.
\end{lemma}

This is a special case of Proposition \ref{prop:transpose}, whose proof we defer to Section \ref{sec:glueseqbi}.

\subsection{Path DGA}

\begin{definition}
The {\em path DGA} of the bigraph $\tch$ over the field $\mathbb{F}$ consists of a graded algebra $\tca$ equipped with a differential $\tilde{\delta} : \tca \to \tca$.  The algebra $\tca$ has underlying vector space $\mathbb{F}\langle\tcp\rangle$, with the subspace $\tca^k$ in grading $k$ spanned by the set $\tcp^k$ of paths of weight $k$.  Multiplication is given by concatenation as before.

The value of $\tilde\delta$ on an interior edge $\tve$ is the sum of all refinements of $\tve$ to a path of weight $2$.  The value of $\tilde\delta$ on a boundary edge $\bve$ is the sum of all refinements of $\bve$ to a boundary path of length 2.
\end{definition}
 
The differential $\tilde\delta$ has degree 1, since a boundary refinement of a boundary edge $\bve$ has length 2 if and only if it has weight $w(\bve) + 1$ by \eqref{genweight}.  We may define $\tilde\delta$ directly on a path $\tilde\rho = ((V,E), \tilde\rho)$ by
\begin{align}
\label{eqn:tdelta}
\tilde\delta \tilde{\rho} = \sum_{v \in V^0} \sum_{\tr' \in \tcp^2} r((\tr, v), \tr') + \sum_{v \in V^1} \sum_{\br' \in \bcp^2} r((\tr, v), \br').
\end{align}
We now define the adjoint map $\tilde\partial$ with respect to the symmetric, bilinear pairing $$\langle \, \cdot \,, \, \cdot \, \rangle : \tca \times \tca \to \mathbb{F}$$ with $\langle \tr_1, \tr_2\rangle = 1$ if and only if $\tr_1 = \tr_2$.

\begin{definition}
\label{def:partialb}
The map $\tilde\partial: \tca \to \tca$ is defined on the path $\tilde\rho$ with underlying tree $T=(V,E)$ by
\begin{align}
\label{eqn:tpartial}
\tilde\partial \tr = \sum_{x \in X} \tilde{q}(\tr, x)
\end{align}
where $\tilde{q}(\tau, x)$ is defined to be zero on invalid input.
\end{definition}

\begin{theorem}
\label{thm:deltatildesq}
We have $$\langle \tilde\partial a, b \rangle = \langle a, \tilde\delta b \rangle$$
for all $a, b \in \tca$.  Furthermore, $\tilde\partial^2 = 0$ and $\tilde\delta^2 = 0$.
\end{theorem}

\begin{proof}
For the first statement, by linearity it suffices to verify that $$\langle \tilde\partial \tr_1, \tr_2 \rangle = \langle \tr_1, \tilde\delta \tr_2 \rangle$$ for all paths $\tr_1$ and $\tr_2$. The left-hand side counts all 1-step gluing sequences $\tr_1 \xrightarrow{x_1} \tr_2$.  The right-hand side counts all ways to refine an interior edge of $\tr_2$ to a path of weight 2 to obtain $\tr_1$, plus all ways to refine a boundary edge of $\tr_2$ to a boundary path of length 2 to obtain $\tr_1$.  These counts coincide by Corollary \ref{cor:legalweight2}.

To prove $\tilde\partial^2 = 0$, it suffices to show that $\langle \tilde\partial^2\tr_1, \tr_3\rangle$ vanishes for all paths $\tr_1$ and $\tr_3$.  Indeed, the term $\langle \tilde\partial^2\tr_1, \tr_3\rangle$ counts the number of 2-step gluing sequences from $\tr_1$ to $\tr_3$, which is even by Lemma \ref{lem:even}.  Now $\tilde\delta^2 = 0$ by adjointness as before.
\end{proof}

\subsection{Tree-like higraphs and bigraphs}
\label{sec:treelike}

The following notions for abstract paths readily carry over to bigraphs as well.

\begin{definition}
A higraph or bigraph is {\em tree-like} if the target of each edge consists of exactly one vertex.  A path in a bigraph is {\em good} if each end is interior or unstable.
A break is {\em good} if it is interior or unstable.
\end{definition}

Clearly the bigraph $\tch$ is tree-like if and only if the underlying higraph $\ch$ is tree-like.  Let $\tcg^k$ be the set of good paths in $\tch$ of weight $k$.  Let $\tcg^k_m$ be the set of good paths in $\tch$ of weight $k$ with $m$ good breaks.  Note that $\tcg^1 = \tcg^1_0$ and $\tcg^2 = \tcg^2_1 \cup \tcg^2_0$ by \eqref{goodweight}; the forms of these paths are shown in Figure \ref{fig:treebreak}.

\begin{definition}
We define the element $\tilde{D} \in \tca$ by $$\tilde{D} = \sum_{\tilde\rho \in \tcg^1} \tilde\rho.$$
\end{definition}

\begin{theorem}
\label{thm:bistructure}
If $\ch$ is tree-like, then the element $\tilde{D}$ satisfies the structure equation:
$$\tilde\delta \tilde D = \tilde D \circ \tilde D.$$
\end{theorem}
\begin{proof}
We claim
\begin{align}
\label{bistructure2}
\tilde\delta \tilde D = \sum_{\tr\in \tcg^2_1} \tr + \sum_{\tr\in \tcg^2_0} 2\tr = \sum_{\tr\in \tcg^2_1} \tr = \tilde D \circ \tilde D.
\end{align}
The left-hand inequality holds by Corollary \ref{cor:onetwoglue}.  The middle equality holds over $\mathbb{F}$.  The right-hand equality holds because each good path of weight 2 with one good break arises as a product of two (not necessarily distinct) good paths of weight 1 in a unique way (as specified by cutting at the good break).
\end{proof}

\subsection{Balanced higraphs and bigraphs}
The reader may check that Theorem \ref{thm:bistructure} is always false if $\ch$ is not tree-like.  On the other hand, there is a duality at play, which switches the roles of unstable and stable vertices.  Dual to tree-like higraphs are {\em rake-like} higraphs, where the source of each edge consists of a single vertex.  If $\ch$ is rake-like, and we re-define good breaks to be interior or stable ones, then Theorem \ref{thm:bistructure} continues to hold.  More generally, we can independently define the notion of good at each vertex of $\ch$, provided each edge of $\ch$ has a compatible shape.

\begin{definition}
A {\em vertex assignment} on a higraph $\ch = (\cv, \ce)$ is a map $\psi: \cv \to \{1,2\}$ such that
\begin{align}
\label{balanced}
|s(\ve) \cap \psi^{-1}(2)| + |t(\ve) \cap \psi^{-1}(1)| = 1
\end{align}
for every edge $\ve \in \ce$.  In this case, we say that $\ch$ is {\em balanced} by $\psi$.  The set of $\psi$-good vertices of $\tch$ is given by
$$\tcv_\psi = (\psi^{-1}(1) \times \{o,u\}) \cup  (\psi^{-1}(2) \times \{o,s\}) \subset \tcv.$$
A path $\tr$ is $\psi$-good if each end is $\psi$-good.  A break $e$ of $\tr$ is $\psi$-good if $\tr(e)$ is $\psi$-good.  Let $\tcp^k(\psi)$ denote the set of $\psi$-good paths of weight $k$.  Let $\tcp^k_m(\psi)$ denote the set of $\psi$-good paths of weight $k$ with $m$ $\psi$-good breaks.
\end{definition}

Note if $\psi \equiv 1$ then $\tcp^k_m(\psi) = \tcg^k_m$, and $\psi$ balances $\ch$ if and only if $\ch$ is tree-like.  If $\psi \equiv 2$, then $\psi$-good breaks are interior or stable ones, and $\psi$ balances $\ch$ if and only if $\ch$ is rake-like.

\begin{remark}
The proof of the monopole surgery exact triangle is modeled by a higraph with $\psi = 1$ for three vertices and $\psi = 2$ for four vertices, as explained in Figure \ref{fig:exactTriangle}.
\end{remark}

\begin{definition}
We define the element $\tilde{D}_\psi \in \tca$ by $$\tilde{D}_\psi = \sum_{\tilde\rho \in \tcp^1(\psi)} \tilde\rho.$$
\end{definition}

\begin{corollary}
If $\ch$ is balanced by $\psi$, then the element $\tilde{D}$ satisfies the structure equation:
$$\tilde\delta \tilde D_\psi = \tilde D_\psi \circ \tilde D_\psi.$$
\end{corollary}

\begin{proof}
Given $\ch = (\cv, \ce)$ balanced by $\psi$, we set $\ch' = (\cv, \ce')$ with $\ce$ in bijection with $\ce'$ via
$$\ve = (I,J) \mapsto \ve' = \left(\left(I \cap \phi^{-1}(1)\right) \cup \left(J \cap \phi^{-1}(2)\right), \left(J \cap \phi^{-1}(1)\right) \cup \left(I \cap \phi^{-1}(2)\right)\right).$$
In other words, we flip each end with $\psi = 2$ from source to target or vice versa.
The equation
$$\tilde\delta \tilde D = \sum_{\tr\in \tcp^2_1(\psi)} \tr + \sum_{\tr\in \tcp^2_0(\psi)} 2\tr = \sum_{\tr\in \tcp^2_1(\psi)} \tr = \tilde D \circ \tilde D$$
on $\tch$ is equivalent to equation \eqref{bistructure2} on $\tch'$.  The latter holds since $\ch'$ is tree-like by \eqref{balanced}.
\end{proof}

\subsection{Boundary bigraphs}
\label{sec:bar}

We briefly return to boundary bigraphs.

\begin{definition}
The {\em path DGA} of the boundary bigraph $\bch$ over the field $\mathbb{F}$ consists of a graded algebra $\bca$ equipped with a differential $\bar{\delta} : \bca \to \bca$.  The algebra $\bca$ has underlying vector space $\mathbb{F}\langle\bcp\rangle$, with the subspace $\bca^k$ in grading $k$ spanned by the set $\bcp^k$ of paths of length $k$.  Multiplication is given by concatenation.

The value of $\bar\delta$ on a boundary edge $\bve$ equals the sum of all refinements of $\bve$ to a boundary path of length 2.  In particular, $\bca$ sits inside $\tca$ as a differential algebra.
\end{definition}

The arguments in Section \ref{sec:higraphs} for higraphs carry over essentially without change to boundary bigraphs.  In particular, $\bar\delta$ is indeed a differential and the element 
$$\bar{D} = \sum_{\br \in \bcp^1} \br$$
of $\bca$ satisfies the structure equation
$$\bar\delta \bar D = \bar D \circ \bar D$$
without restriction on the shape of edges.

\subsection{Gluing sequences}
\label{sec:glueseq}

We now take a closer look at gluing sequences in order to prove Lemma \ref{lem:factorial1} and Lemma\ref{lem:even}.  A path in a higraph with two breaks has a $2$-step gluing sequence.  However, there exist paths in a bigraph, such as path $(d)$ in Figure \ref{fig:treebreak}, with two available gluing moves but no $2$-step gluing sequence.  The subtlety is that the breaks in a gluing sequence live on different paths.  At least for bigraphs, it is not immediately clear how to go about reordering them.  The key will be to pull the breaks or markers in a gluing sequence up to breaks on the initial path.  Though the higraph case is straightforward, we consider it in detail as a model for the bigraph case.
 
\subsubsection{Higraphs}
\label{sec:glueseqhi}
Let $\rho_1$ be a path in a higraph $\ch$.   We will define a bijection between $k$-tuples of distinct breaks on $\rho_1$ and $k$-step gluing sequences on $\rho_1$.  We relate these notions by populating a lower triangular matrix $A$ whose $(i,j)$-entry $e^i_j$ is a break on $\rho_j$, with $\rho_j = q(\rho_{j-1},e^{j-1}_{j-1})$.  Here we use the bijection $\pi$ on arrows as defined following \eqref{pi}.

Let $(e^1_1, e^2_1, \dots e^k_1)$ be a $k$-tuple of distinct breaks on a path $\rho_1$, thought of as the first column of $A$.  We populate $A$ one column at a time from left to right as follows.  For each $j$ from $2$ to $k$, we set
$$\rho_j = q(\rho_{j-1},e^{j-1}_{j-1})$$
and
$$e^i_j = \pi(e^i_{j-1})$$
for $j \leq i \leq k$.  The diagonal of $A$ is a $k$-step gluing sequence $(e^1_1, e^2_2, \dots, e^k_k)$ on $\rho_1$.

Let  $(e^1_1, e^2_2, \dots e^k_k)$ be a $k$-step gluing sequence on $\rho_1$, thought of as the diagonal of $A$.  We populate $A$ one column at a time from right to left as follows.
For each $j$ from $k$ down to $2$, we have
$$\rho_j = q(\rho_{j-1}, e^{j-1}_{j-1})$$
and can thus set
$$e^i_{j-1} = \pi^{-1}(e^i_j)$$
for $j \leq i \leq k$.  The first column of $A$ is a $k$-tuple $(e^1_1, e^2_1, \dots e^k_1)$ of distinct breaks on $\rho_1$.

\begin{proposition}
\label{prop:biject1}
Let $\rho_1$ be a path in a higraph.  The map
$$(e^1_1, e^2_1, \dots e^k_1) \mapsto (e^1_1, e^2_2, \dots e^k_k)$$
gives a bijection between $k$-tuples of distinct breaks on $\rho_1$ and $k$-step gluing sequences on $\rho_1$.

Furthermore, if two $k$-tuples differ by a permutation, then the corresponding $k$-step gluing sequences terminate in the same path $\rho_{k+1}$.
\end{proposition}
\begin{proof}
The map is a bijection as we have also defined its inverse above.

For the latter statement, consider two $k$-tuples $(e^1_1, e^2_1, e^3_1, \dots e^k_1)$ and $(f^1_1, f^2_1, f^3_1, \dots, f^k_1)$ of distinct breaks on $\rho_1$ which differ by a permutation, which we may assume is a transposition of neighbors.  We may reduce further to the case of the transposition of the first two entries by induction on $k$, since if $e^1_1 = f^1_1$ then $(e^2_2, \dots, e^k_2)$ and $(f^2_2, \dots, f^k_2)$ are $(k-1)$-tuples of distinct breaks on $\rho_2 = q(\rho_1, e^1_1) = q(\rho_1, f^1_1)$.  Note that the terminal path of $(e^2_2, \dots, e^k_2)$ is the same as that of $(e^1_1, e^2_1, e^3_1, \dots e^k_1)$ by definition, and similarly for $(f^2_2, \dots, f^k_2)$.

So we suppose $e^1_1 = f^1_2$, $e^1_2 = f^1_1$, and $e^i_1 = f^i_1$ for $3 \leq i \leq k$.  Let $T^e_3$ and $T^f_3$ be the trees underlying the paths $q(q(\rho_1,e^1_1),e^2_2))$ and $q(q(\rho_1,f^1_1),f^2_2))$.  When $e^1_1$ and $e^2_1$ are adjacent, these trees result from contracting $T_1$ along the same two-arrow subtree $T'$, and we therefore have an isomorphism $f: T^e_3 \to T^f_3$ which identifies $\pi_{e^2_2}(\pi_{e^1_1} (v))$ with $\pi_{f^2_2}(\pi_{f^1_1}(v))$ for each $v \in V_1$, and identifies the arrow $e^i_3$ with $f^i_3$ for $3 \leq i \leq k$.  Furthermore, the map $f$ is an isomorphism of paths, since for both paths the edge over the node $v^*$ to which $T'$ contracts has the same source and target as the subpath of $\rho_1$ supported by $T'$.  Finally, if $e^1_1$ and $e^2_1$ are not adjacent, then the isomorphism of paths is clear.

Therefore the paths in the gluing sequences $(e^1_1, e^2_2, e^3_3, \dots, e^k_k)$ and $(f^1_1, f^2_2, f^3_3, \dots, f^k_k)$ are identified from $\rho_3$ to $\rho_{k+1}$.
\end{proof}

\begin{proof}[Proof of Lemma \ref{lem:factorial1}]
Via Proposition \ref{prop:biject1}, the set of $k$-step gluing sequences between any two fixed paths admits a free action of the group of permutations on $k$ elements.
\end{proof}

\subsubsection{Bigraphs}
\label{sec:glueseqbi}

We now take an analogous, though more involved, approach to bigraphs.  We will define a bijection between certain $k$-tuples of markers on $\rho_1$ and $k$-step gluing sequences on $\rho_1$.  Recall the notion of a marker from Definition \ref{def:marker}.

\begin{definition}
A marker $x$ on $\tr_1$ is {\em valid} if  $\tilde q(\tr_1, x)$ is defined.  A $k$-tuple $(x^1_1,\dots,x^k_1)$ of markers on $\tr_1$ is {\em valid} if its coordinates are valid and distinct.
\end{definition}

Suppose $x^1_1$ is a valid marker on $\tr_1$, and set $\tr_2 = \tilde q(\tr_1, x^1_1)$.  Let $T' = (V', E')$ be the subtree of $T_1$ that is contracted under the gluing move, and let $v^*$ be the node in $T_2$ onto which $T'$ contracts.  Recall that $\pi$ denotes both the map $\pi : V_1 \to V_2$ which is injective on $V_1 - V'$ with image $V_2 - \{v^*\}$, and the bijection $\pi : E_1 - E' \to E_2$.  We thus have maps $\pi : X_1 - E' \to X_2$ and $\pi^{-1} : X_2 - \{v^*\} \to X_1$ which are mutual inverses whenever their composition is defined.

We define maps $\tilde\pi : X_1 - E' \to X_2$ and $\tilde\pi^{-1} : X_2 \to X_1$ as follows: \\
\begin{align*}
\tilde\pi(x) &=
\left\{\begin{array}{rll}
\pi(s(x)) &\text{if $x$ is a boundary arrow and } t(x) = x^1_1, & \hspace{25.1mm} \text{(1a)} \\
\pi(t(x)) &\text{if $x$ is a boundary arrow and } s(x) = x^1_1, & \hspace{25.1mm}  \text{(1b)} \\
\pi(x) &\text{otherwise}. & \hspace{25.1mm}  \text{(1c)}
\end{array}\right. \\
& \\
\tilde\pi^{-1}(y) &=
\left\{\begin{array}{rll}
(\pi^{-1}(y),x^1_1) &\text{if $y$ is a boundary node and $(y,\pi(x^1_1)) \in E_2$}, &  \hspace{10mm} \text{(2a)}\\
(x^1_1,\pi^{-1}(y)) &\text{if $y$ is a boundary node and $(\pi(x^1_1), y) \in E_2$}, & \hspace{10mm} \text{(2b)}\\
x^1_1 &\text{if } y = \pi(x^1_1), & \hspace{10mm} \text{(2c)} \\
t(x^1_1) &\text{if } y = \pi(t(x^1_1)) \text{ and $x^1_1$ is unstable}, & \hspace{10mm} \text{(2d)} \\
s(x^1_1) &\text{if } y = \pi(s(x^1_1)) \text{ and $x^1_1$ is stable or interior}, & \hspace{10mm} \text{(2e)} \\
\pi^{-1}(y) &\text{otherwise}. & \hspace{10mm} \text{(2f)}
\end{array}\right.
\end{align*}

Note that in (1a) and (1b), $x^1_1$ is a marker for Move Ib.  In (2a), (2b), and (2c), $x^1_1$ is a marker for Move Ib.  In (2d), $x^1_1$ is a marker for Move II, while in (2e), $x^1_1$ is a marker for Move II or Move Ia.  In (2c), (2d), and (2e), $y$ coincides with the node $v^*$.

The following facts will be useful in the proof of Lemma \ref{lem:valid} below: \\
{\em Fact 1:} The maps $\pi|_{X_1 - V' - E'}$ and $\pi^{-1}|_{X_2 - \{v^*\}}$ are mutual inverses and satisfy $\tr_2 \pi = \tr_1$ and \\
\indent \hspace{8mm} $\tr_2 = \tr_1 \pi^{-1}$. In particular, these maps preserve abstract vertex and edge type. \\
{\em Fact 2:} If $x^1_1$ is a valid marker for Move Ib, then $x^1_1$ is boundary while $\pi(x^1_1)$ is interior and \\
\indent \hspace{8mm} thus not valid. \\

\begin{lemma}
\label{lem:valid}
The maps $\tilde\pi : X_1 - E' \to X_2$ and $\tilde\pi^{-1} : X_2 \to X_1$ are well-defined.  Furthermore:
(i)  If $y \in X_2$, then $\tilde\pi(\tilde\pi^{-1}(y)) = y$. If $y$ is valid, then $\tilde\pi^{-1}(y)$ is distinct from $x^1_1$ and valid. \\
(ii) If $x \in X_1$ is distinct from $x^1_1$ and valid, then $\tilde\pi(x)$ is defined and $\tilde\pi^{-1}(\tilde\pi(x)) = x$.
\end{lemma}
\begin{proof}
The cases within each definition are mutually exclusive.  The map $\tilde\pi$ is well-defined because its domain coincides with that of $\pi$.  For the map $\tilde\pi^{-1}$, there are two issues to check.  First, Fact 1 implies that $(\pi^{-1}(y), x^1_1)$ is an arrow of $T_1$ in (2a) and that $(x_1^1, \pi^{-1}(y))$ is an arrow of $T_1$ in (2b).  Second, if $y = v^*$, then $\tilde\pi^{-1}(y)$ is defined by (2c) if $x^1_1$ is a node and by (2d) or (2e) when $x^1_1$ is an arrow.

For the first claim of (i), note that the image of $\tilde\pi^{-1}$ is disjoint from $E'$, so that $\tilde\pi(\tilde\pi^{-1}(y))$ is well-defined.  Cases (1a) and (2a) define mutual inverses, as do (1b) and (2b).  Case (1c) inverts (2c), (2d), (2e), and (2f) by definition (thought not necessarily the other way around).

Now suppose that $y \in X_2$ is valid.  If $x^1_1$ is a node, then $v^* = \pi(x_1^1)$ is an interior node and thus not valid.  Nor is $x^1_1$ in the image of $\pi^{-1}$.  So $\tilde\pi^{-1}(y)$ is distinct from $x^1_1$.

To show that $\tilde\pi^{-1}(y)$ is valid, we consider each case of the definition of $\tilde\pi^{-1}$ in turn.

Case (2a): $y$ is a boundary node with $(y, \pi(x_1^1)) \in E_2$.  Then $\tilde\pi^{-1}(y)$ is the arrow $(\pi^{-1}(y), x^1_1)$ which runs between two boundary nodes by Facts 1 and 2 and is thus valid.

Case (2b): This case is similar (2a).

Case (2c): This case does not occur when $y$ is valid by Fact 2.

Case (2d): $y = v^*$ and $x^1_1$ is an unstable arrow, with $\tilde\pi^{-1}(y) = t(x^1_1)$.  Since $x^1_1$ is a valid boundary arrow, $t(x^1_1)$ is a boundary node.  The edge $\tr_2(y)$ results from applying Move II to the subpath $\tr_1(s(x^1_1))\tr_1(t(x^1_1))$.   Furthermore, by Fact 2, the interior neighborhoods of $y$ and $t(x^1_1)$ are identified.  Since $y$ is valid and the break $\tr_1(x^1_1)$ is unstable, we conclude that $t(x^1_1)$ is valid (while $s(x^1_1)$ is not). 

Case (2e): This case is similar to Case 4 when $x^1_1$ is stable.  Note that if $y = v^*$ is valid, then $x^1_1$ cannot be an interior arrow.

Case (2f): If $y$ is a valid arrow then its ends are both interior or both boundary, and it suffices to show that the same is true of $\tilde\pi^{-1}(y) = \pi^{-1}(y)$.  This follows from Fact 1 if $y$ is not bounded by the interior node $v^*$.  On the other hand, if $y$ starts or ends at $v^*$, then both ends of $y$ are interior nodes, and thus $y$ is an interior arrow.  Hence $\pi^{-1}(y)$ is also an interior arrow by Fact 1, and we conclude that its ends are both interior as well.  If $y$ is a valid node and does not fall into (2a) through (2e), then Star$^o(y)$ does not include $v^*$ and thus Star$^o(\pi^{-1}(y))$ is legal as well by Fact 1.  Therefore $\pi^{-1}(y)$ is also a valid node.  This completes (i).

For (ii), we suppose $x \in X_1$ is distinct from $x^1_1$ and valid.  If $x^1_1$ is an arrow, then $E' = \{x^1_1\}$ so $x \notin E'$.  If $x^1_1$ is a node, then $E'$ consists of arrows between the boundary node $x^1_1$ and interior nodes.  These arrows are not valid, so $x \notin E'$.  Therefore $\tilde\pi(x)$ is defined.

Finally, we show that $\tilde\pi^{-1}(\tilde\pi(x)) = x$. As in (i), cases (1a) and (2a) are mutual inverses, as are (1b) and (2b).  So from now on we assume $x$ falls into (1c), with $\tilde\pi(x) = \pi(x)$.  If $\pi(x) = \pi(x^1_1)$ as in (2c), then $x = x^1_1$ since the nodes adjacent to $x^1_1$ are not valid.  If $\pi(x)$ falls into (2d) or (2e), then $x^1_1$ is a boundary arrow bounded by the valid node $x$ and $\pi(x) = \pi(s(x^1_1)) = \pi(t(x^1_1))$.  If $x^1_1$ is unstable (resp., stable) then we must have $x = t(x^1_1)$ (resp., $ x = s(x^1_1)$) in order for Star$^o(x)$ to be legal.  This is consistent with the definition of $\tilde\pi^{-1}$ in (1d) and (1e).  Finally, if $\pi(x)$ falls into (1f), then $\tilde\pi^{-1}(\tilde\pi(x)) = \pi^{-1}(\pi(x)) = x$ by Fact 1.
\end{proof}

Paralleling our strategy for higraph, we now populate a lower triangular matrix whose $(i,j)$-entry $x^i_j$ is a valid marker on $\tr_j$, with $\tr_j = \tilde q(\tr_{j-1},x^{j-1}_{j-1})$.

Let $(x^1_1, x^2_1, \dots x^k_1)$ be a $k$-tuple of markers on a path $\tr_1$.  For each $j$ from $2$ to $k$, we set
$$\tr_j = \tilde q(\tr_{j-1},x^{j-1}_{j-1})$$
and
$$x^i_j = \tilde\pi(x^i_{j-1})$$
for $j \leq i \leq k$.  This inductive definition may break down if some $x^i_j$ is not a valid marker.
\begin{definition}
A $k$-tuple $(x^1_1, x^2_1, \dots,x^k_1)$ of valid markers on $\tr_1$ is {\em completely valid} if $x^i_j$ is defined and valid for all $1 \leq j \leq i \leq k$.
\end{definition}
Thus, a completely valid $k$-tuple $(x^1_1, x^2_1, \dots,x^k_1)$ on $\tr_1$ determines a $k$-step gluing sequence $(x^1_1, x^2_2, \dots, x^k_k)$ on $\tr_1$.

Let  $(x^1_1, x^2_2, \dots x^k_k)$ be a $k$-step gluing sequence on $\tr_1$. For each $j$ from $k$ down to $2$, we have
$$\tr_j = \tilde q(\tr_{j-1}, x^{j-1}_{j-1})$$
and can thus set
$$x^i_{j-1} = \tilde\pi^{-1}(x^i_j)$$
for $j \leq i \leq k$.  This determines a completely valid $k$-tuple $(x^1_1, x^2_1, \dots x^k_1)$ of markers on $\tr_1$ by Lemma \ref{lem:valid}, which also implies that the above maps are inverses.  We conclude:

\begin{proposition}
\label{prop:biject2}
Let $\tr_1$ be a path in a bigraph.  The map $$(x^1_1, \dots, x^k_1) \mapsto (x^1_1, \dots, x^k_k)$$ gives a bijection between completely valid $k$-tuples of markers on $\tr_1$ and $k$-step gluing sequences on $\tr_1$.
\end{proposition}

In a bigraph, permutations of a completely valid $k$-tuple need not be completely valid.  Indeed, in Figure \ref{fig:eightPaths} each set of 3 valid makers on $\tr_1$ is completely valid for only two out of six possible orderings, and there are precisely eight $3$-step gluing sequences. However, Lemma \ref{lem:factorial1} does extend from higraphs to bigraphs when $k=2$.

\begin{proposition}
\label{prop:transpose}  Let $(x^1_1, x^2_1)$ be a completely valid $2$-tuple for a path in a bigraph. \\
(i) The transposed $2$-tuple $(x^2_1, x^1_1)$ is completely valid. \\
(ii) The $2$-step gluing sequences for $(x^1_1, x^2_1)$ and $(x^2_1, x^1_1)$ terminate in the same path. \\
(iii) For $k \geq 2$, the number of $k$-step gluing sequences between any two fixed paths is even.
\end{proposition}

\begin{proof}
For (i), Let $(x^1_1, x^2_1)$ be a valid $2$-tuple of markers on $\tr_1$ and let $(y^1_1, y^2_1)$ be its transposition, so $y^1_1 = x^1_2$, $y^1_2 = x^1_1$.
By definition, we have $x^2_2 = \tilde\pi_{x^1_1}(x^2_1)$ and $y^2_2 = \tilde\pi_{y^1_1}(y^2_1)$, where the subscript on $\tilde\pi$ specifies the marker at which we glued.  We claim that the marker $x^2_2$ on $\tilde q(\tr_1,x^1_1)$ is valid if and only if the marker $y^2_2$ on $\tilde q(\tr_1,y^1_1)$ is valid.

By Fact 1, the claim is clear if $x_1^1$ and $x^2_1$ do not interact; that is, if the subtrees of $T_1$ that are contracted by gluing at $x_1^1$ and $x^2_1$ are disjoint.  There are five ways in which two valid markers can interact.  These are pictured in columns $(a)$ through $(e)$ of Figure \ref{fig:glueTwiceB}.    One can check that $x_2^2$ and $y^2_2$ are always both valid in all cases but $(e)$.  In column $(e)$, $x^1_1$ is a boundary node, $x^2_1$ is a boundary arrow, and $t(x^2_1) = x^1_1$ or $s(x^2_1) = x^1_1$.  In this case, 
Figures \ref{fig:glueTwice} and \ref{fig:glueTwiceNope} show that $x^2_2$ is valid if and only if $y^2_2$ is valid.

For (ii), let $T^x_3$ and $T^y_3$ be the trees underlying the paths $\tilde q(\tilde q(\tr_1,x^1_1),x^2_2))$ and $\tilde q(\tilde q(\tr_1,y^1_1),y^2_2))$.  In cases $(a)$ through $(e)$ of Figure \ref{fig:glueTwiceB}, one can check that the trees result from contracting the same subtree $T'$ of the tree $T_1$, and we therefore have an isomorphism $f: T^x_3 \to T^y_3$ which identifies $\pi_{x^2_2}(\pi_{x^1_1} (v))$ with $\pi_{y^2_2}(\pi_{y^1_1}(v))$ for each $v \in V_1$.  The most interesting case $(e)$ is verified in Figure \ref{fig:glueTwice}.  Furthermore, in each case, the map $f$ is an isomorphism of paths, since for both paths the edge over the node $v^*$ to which $T'$ contracts has the same source and target as the subpath of $\rho_1$ supported by $T'$.  Finally, if $x^1_1$ and $x^2_2$ do not interact, then the isomorphism of paths is clear.

For (iii), note that we have shown that the set of $2$-step gluing sequences between any two fixed paths admits a fixed-point free involution, given by transposition on the corresponding $2$-tuple of valid markers.  For $k > 2$, the set of $k$-step gluing sequences admits a fixed-point free involution as well, since each $k$-step gluing sequence factors uniquely as a $(k-2)$-step gluing sequence followed by a $2$-step gluing sequence.  We conclude that these sets contain an even number of elements.
\end{proof}

\begin{figure}[hp]
\centering
\includegraphics[width=156mm]{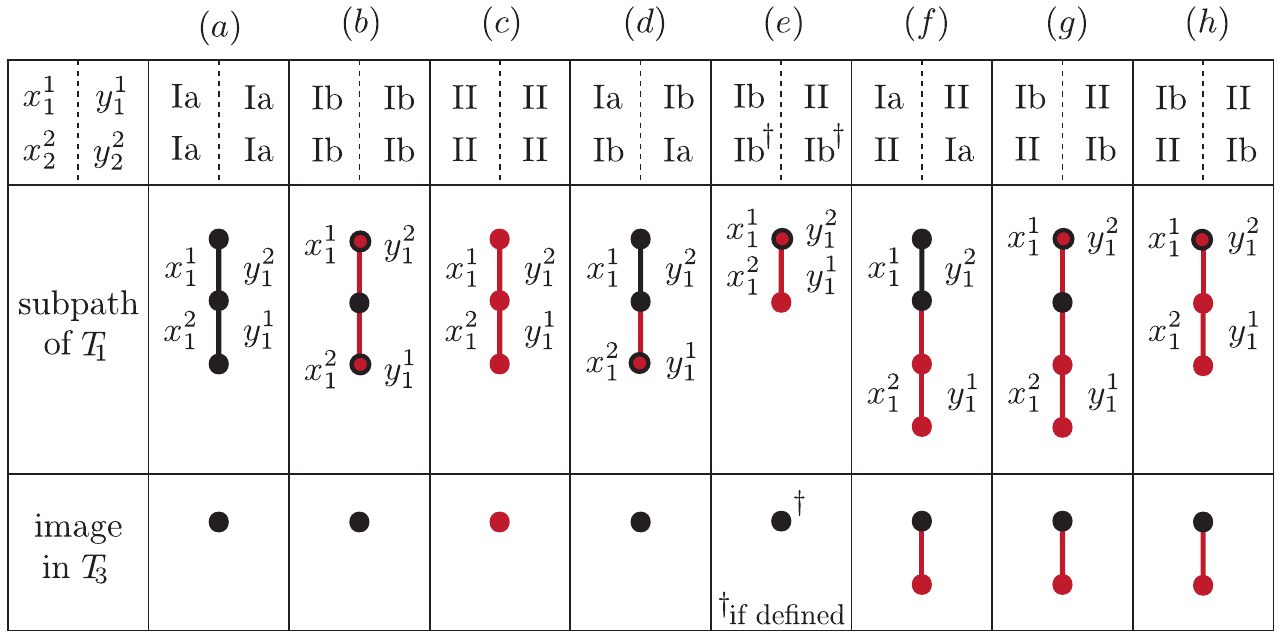}
\caption{To a $2$-tuple $(x^1_1, x^2_1)$ of valid markers, we may associate the minimal undirected subpath of $T_1$ which includes both markers.  Several such subpaths are depicted in the second row above.  There are exactly five (undirected) subpaths that may arise when the contracting trees of $x^1_1$ and $x^2_1$ overlap.  These are depicted in columns $(a)$ through $(e)$.  In the first row, we depict the gluing type of the markers $x^1_1$, $x^2_2$, $y_1^1$, and $y^2_2$, where $(y^1_1, y^2_1)$ is the transposition of $(x^1_1, x^2_1)$.  In all columns except $(e)$, the markers $x^2_2$ and $y^2_2$ are always valid.  In column $(e)$, $x^2_2$ is valid if and only if $y^2_2$ is valid, as shown in Figures \ref{fig:glueTwice} and \ref{fig:glueTwiceNope}.  If $x^2_2$ and $y^2_2$ are valid, then $\tr_3$ is defined and independent of gluing order, as may be easily verified in each case above.  In the third row, we show the image in $T_3$ of the above subpath of $T_1$.}
\label{fig:glueTwiceB}
\end{figure}

\begin{figure}[p]
\centering
\includegraphics[width=156mm]{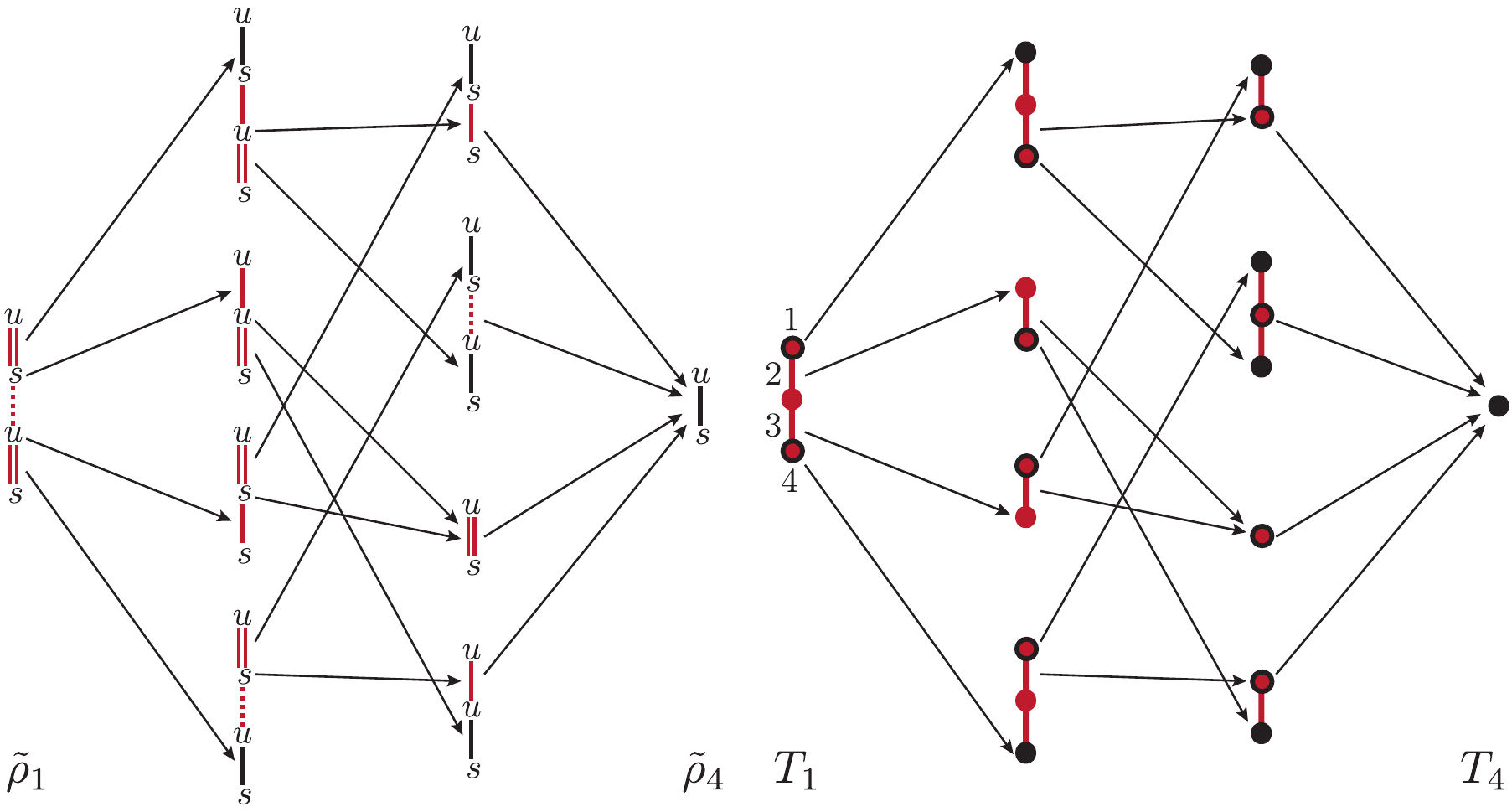}
\caption{The legal path $\tilde\rho_1$ at left has weight $4$, so all $3$-step gluing sequences terminate in the path $\tilde\rho_4$ consisting of a single interior edge.  These eight sequences are shown above, with each arrow positioned to pick out the corresponding gluing move.  The four valid markers of $\tilde\rho_1$ are numbered on the underlying tree $T_1$, with $1$ and $4$ marking Move Ib and $2$ and $3$ marking Move II.  Note that the $2$-tuples $(1,2)$, $(2,1)$, $(3,4)$ and $(4,3)$ fall into the undefined case of column $(e)$ of Figure \ref{fig:glueTwiceB}, as illustrated in Figure \ref{fig:glueTwiceNope}.  The eight completely valid $3$-tuples on $\tilde\rho_1$ are $(1,3,4), (1,4,3), (2,3,4), (2,4,3), (3,1,2), (3,2,1), (4,1,2),$ and $(4,2,1)$.}
\label{fig:eightPaths}
\end{figure}

\begin{figure}[p]
\centering
\includegraphics[width=156mm]{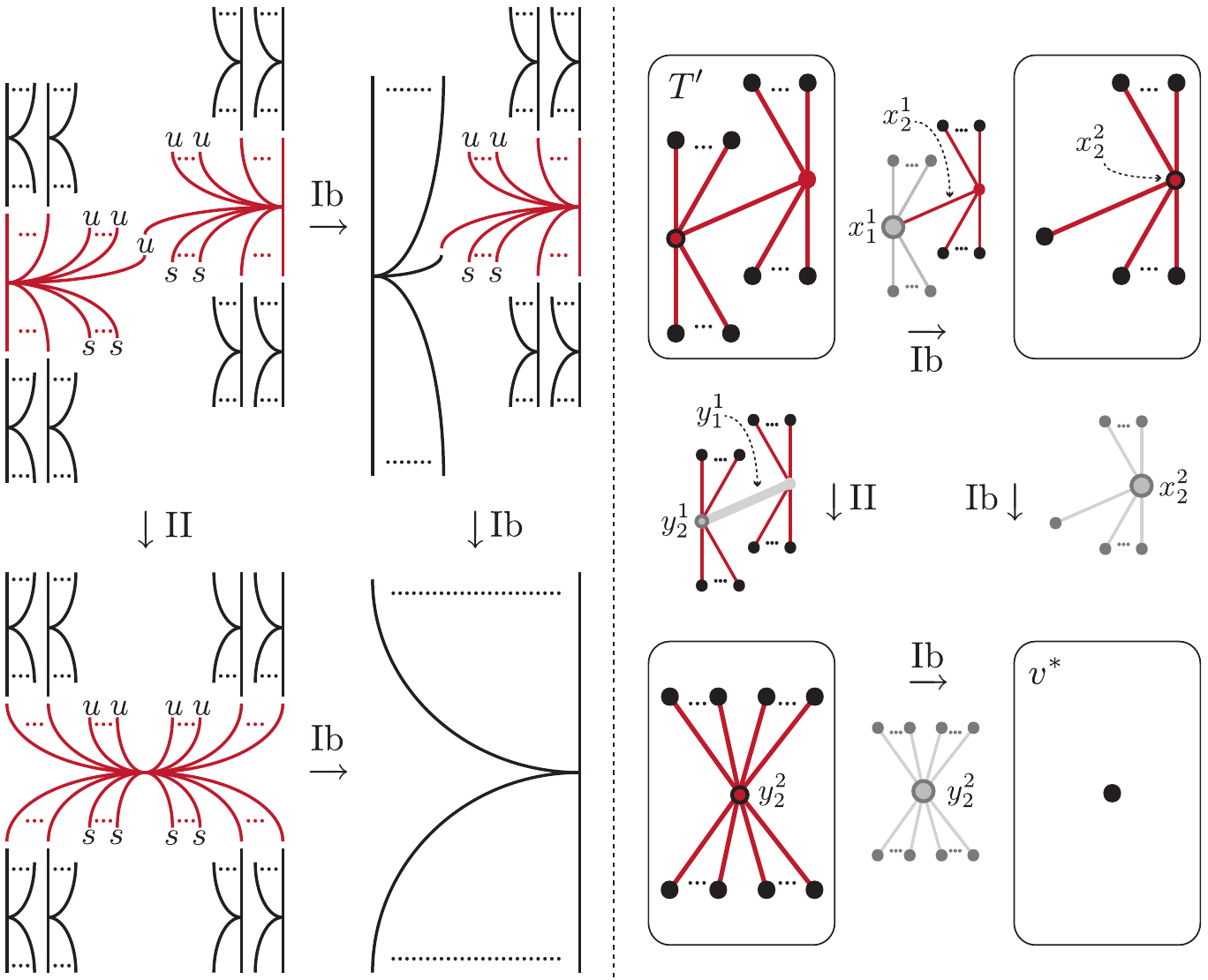}
\caption{In Figures \ref{fig:glueTwice} and  \ref{fig:glueTwiceNope}, we examine the case of column $(e)$ of Figure \ref{fig:glueTwiceB} in detail.   The above diagram illustrates the case when both gluing sequences $(x^1_1, x^2_2)$ and $(y_1^1, y^2_2)$ are defined.  The subtree contracted at each step in shown in gray at right.  Overall, the two sequences contract the same subtree $T'$ of $T_1$ to a single node $v^*$, and therefore terminate in the same path $\tr_3$.  The diagram further illustrates that $x^2_2$ is valid if and only if $y^2_2$ is valid (compare with Figure \ref{fig:glueTwiceNope}).}
\label{fig:glueTwice}
\end{figure}

\begin{figure}[p]
\centering
\includegraphics[width=156mm]{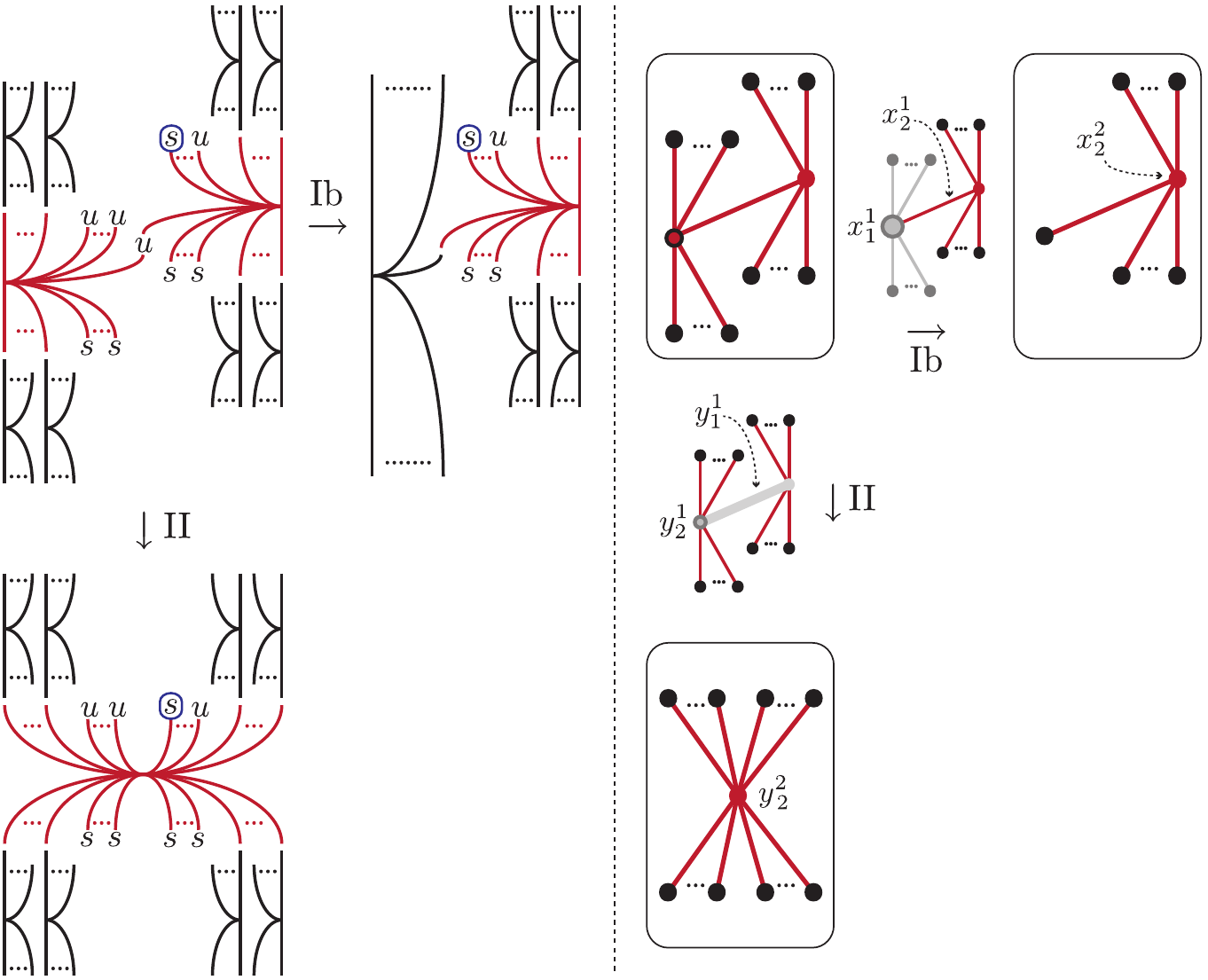}
\caption{In the above diagram, both $x^2_2$ and $y^2_2$ are invalid, so that neither $2$-step gluing sequence is defined.  The difference between Figures \ref{fig:glueTwice} and \ref{fig:glueTwiceNope} amounts to switching the single (encircled) vertex from unstable to stable.}
\label{fig:glueTwiceNope}
\end{figure}

\newpage

\section*{Appendix: Morse homology with boundary}

\subsection*{A. \ \ \, Morse homology on a closed manifold}

We first recall the construction of Morse homology on a closed, smooth manifold $M$.  Fix a Riemannian metric and Morse function $f: M \to \mathbb{R}$.  The gradient flow $\phi_t: M \to M$ with respect to $-\nabla f$ determines an unstable manifold
$$S_a = \{x \in M \, : \, \lim_{t \to -\infty} \phi_t(x) = a\}$$
and stable manifold
$$U_a = \{x \in M \, : \, \lim_{t \to \infty} \phi_t(x) = a\}$$
for each critical point $a$.  The Morse index $\lambda_a$ of $a$ is the dimension of $U_a$.  We insist that $f$ satisfy the {\em Morse-Smale} condition:
\begin{itemize}
\item $U_a$ and $S_b$ intersect transversely in $M$ for all critical points $a$ and $b$.
\end{itemize}
The manifold $U_a \cap S_b$ has dimension $\lambda_a - \lambda_b$ and is canonically identified with the moduli space of gradient trajectories from $a$ to $b$:
$$\mathcal{M}(a,b) = \{ \gamma : \mathbb{R} \to M \, : \, \frac{d\gamma}{dt} = -\nabla f \circ\gamma, \ \lim_{t \to -\infty} \gamma(t) = a, \ \lim_{t \to +\infty} \gamma(t) = b\}.$$
The space $\mathcal{\breve M}(a,b)$ of unparameterized gradient trajectories is the quotient of $\mathcal{M}(a,b)$ by the $\mathbb{R}$-translation $\gamma(\cdot) \mapsto \gamma(\cdot + s)$, so
\begin{align}
\label{eqn:morsedim}
\mathrm{dim} \ \mathcal{\breve M}(a,b) = \lambda_a - \lambda_b - 1.
\end{align}

The Morse complex $C(M)$ is the $\mathbb{F}$-vector space with basis the set of critical points:
$$C(M) = \bigoplus_{a \in \text{crit}(f)} \mathbb{F}a$$
The differential $\partial$ counts unparameterized gradient trajectories:
$$\langle \partial a, b\rangle = |\mathcal{\breve M}(a,b)|.$$
This notation signifies that the coefficient of $b$ in $\partial a$ is the number of unparameterized gradient trajectories from $a$ to $b$, with the convention $|\cdot|=0$ on infinite sets.  The Morse complex is graded by Morse index and the differential has degree $-1$ by \eqref{eqn:morsedim}.  The Morse homology $H_*(C(M),\partial)$ is isomorphic to the singular homology of $M$.

\subsection*{B. \ \ \, Morse homology on a manifold with boundary}

Now let $M$ be a smooth manifold with boundary. The manifold $M$ may be identified with the quotient of a smooth double $$\tilde M = M \cup_{\partial M} -M$$ by the natural involution $i : \tilde M \to \tilde M$.  An $i$-invariant metric and Morse function on $\tilde M$ determine a metric and Morse function on $M$ by restriction. For such a metric and Morse function on $M$, the gradient field $\nabla f$ is tangent to $\partial M$ along the boundary.  Furthermore, at each critical point on the boundary, the normal vector is an eigenvector of the Hessian $\nabla^2 f$.  A boundary critical point is called {\em stable} or {\em unstable} according to whether the corresponding eigenvalue is positive or negative, respectively.  Hence we may sort the critical points of $f$ into three types: interior, (boundary-)stable, and (boundary-)unstable.  We denote these types by the letters $o$, $s$, and $u$. If a trajectory $\gamma \in \mathcal{\breve M}^\partial(a,b)$ intersects the interior of $M$, then $a$ must be interior or unstable and $b$ must be interior or stable.  On the other hand, if $a$ is stable and $b$ is unstable, then both $U_a$ and $S_b$ are contained in $\partial M$ and the most we can ask is that $U_a$ and $S_b$ intersect transversely as submanifolds of $\partial M$.   Indeed, we insist that $f$ is {\em regular}, which means
\begin{itemize}
\item $U_a$ and $S_b$ intersect transversely in $\partial M$ if $a$ is stable and $b$ is unstable.
\item $U_a$ and $S_b$ intersect transversely in $M$ if $a$ is not stable or $b$ is not unstable.
\end{itemize}
Note that regular is equivalent to Morse-Smale if $\partial M = \emptyset$.

Next we define eight linear maps between the vector spaces $C^o(M)$, $C^s(M)$, and $C^u(M)$ generated by the interior, stable, and unstable critical points, respectively:
\begin{align*}
\partial^o_o: C^o(M) \to C^o(M) & \quad & \bar\partial^s_s: C^s(M) \to C^s(M) \\
\partial^o_s: C^o(M) \to C^s(M) & & \bar\partial^s_u: C^s(M) \to C^u(M) \\
\partial^u_o: C^u(M) \to C^o(M) & & \bar\partial^u_u: C^u(M) \to C^u(M) \\
\partial^u_s: C^u(M) \to C^s(M) &  & \bar\partial^u_s: C^u(M) \to C^s(M)
\end{align*}
The maps $\partial^*_*$ count unparameterized gradient trajectories as before, so
$$\langle \partial^*_* a,b \rangle = |\mathcal{\breve M}(a,b)|.$$
The maps $\bar\partial^*_*$ only count those trajectories which remain in the boundary.  More precisely,
$$\langle \bar\partial^*_* a,b \rangle = |\mathcal{\breve M}^\partial(a,b)|$$
where
$\mathcal{\breve M}^\partial(a,b) = \mathcal{\breve M}(a,b) \cap \left(\{\gamma: \mathbb{R} \to \partial M\}/ \mathbb{R}\right)$.

\begin{figure}[htp]
\centering
\includegraphics[width=125mm]{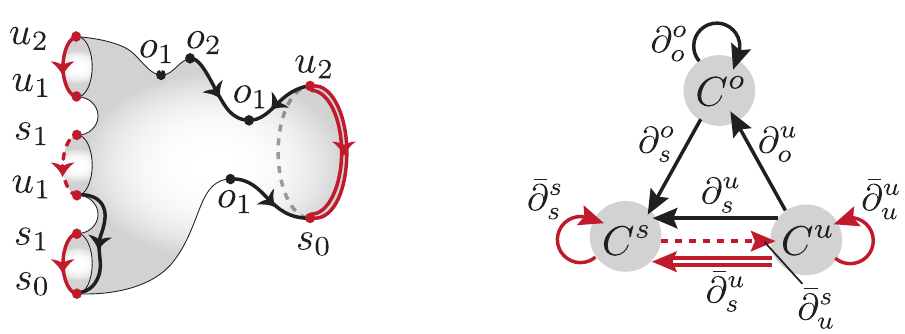}
\caption{Consider the surface at left embedded in $\mathbb{R}^3$ with the induced metric and height function.  Each critical point is labeled by its type and Morse index.  The directed graph at right depicts the vector spaces generated by the interior, stable, and unstable critical points together with eight linear maps between them.  On the surface we have highlighted a single gradient trajectory counted by each map:  four interior trajectories in black and four boundary trajectories in red. The dashed red line from $s$ to $u$ signifies that $\bar\partial^s_u$ is index-preserving.  The doubled red line from $u$ to $s$ signifies that the $\bar\partial^u_s$ lowers index by 2.  The other six maps lower index by 1.}
\label{fig:morseOps}
\end{figure}

The boundary, absolute, and relative Morse complexes of $M$ of are given by:
\begin{align*}
C(\partial M) &= C^s(M) \oplus C^u(M) \ \ \ & C(M) = C^o(M) \oplus C^s(M) \ \ \ \quad & C(M, \partial M) = C^o(M) \oplus C^u(M) \\
\partial_{\partial M} &= \left[\begin{array}{rr}
\ess & \eus\\
\esu & \euu
\end{array}\right] \ \ \
& \partial = \left[\begin{array}{rr}
\doo & \duo\esu\\
\dos & \ess + \dus\esu
\end{array}\right] \ \ \ \quad
& \partial_{M,\partial M} = \left[\begin{array}{rr}
\doo & \duo\\
\esu\dos & \euu + \esu\dus
\end{array}\right]
\end{align*}
The reader should compare the definition of $\partial$ with that of $D$ in Figure \ref{fig:morsePaths}.  Figure \ref{fig:morseExample} illustrates each complex for a particular metric and Morse function on the 2-disk.

\begin{figure}[htp]
\centering
\includegraphics[width=156mm]{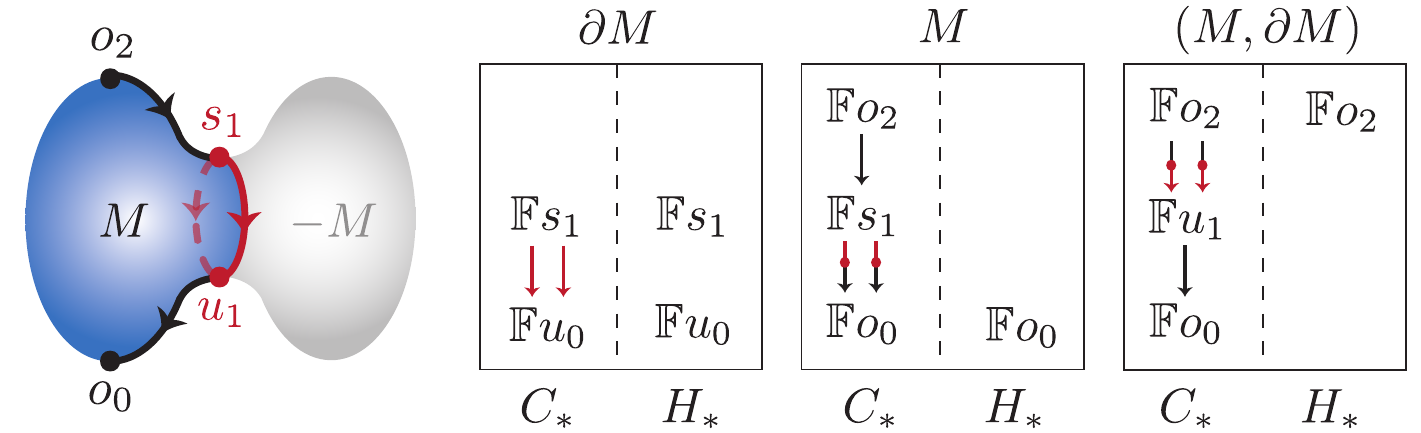}
\caption{At left, the double of the 2-disk $M$ is symmetrically embedded in $\mathbb{R}^3$.  The induced metric and height function on $M$ determine boundary, absolute, and relative Morse complexes over $\mathbb{F}$ (so two arrows are equivalent to zero arrows).  In each case, the Morse homology agrees with the singular homology of the 2-disk as expected.}
\label{fig:morseExample}
\end{figure}

For each map $\partial^*_*$, there is a map $\tilde\partial\delta^*_*$ which counts points in the boundary of the compactification of $\mathcal{\breve M}(a,b)$.  For each map $\bar \partial^*_*$, there is a map $\tilde\partial\delta^*_*$ which counts points in the boundary of the compactification of $\mathcal{\breve M}^\partial(a,b)$.  These maps should be compared with the definition of $\delta$ in Figure \ref{fig:morsePaths}.
\begin{align*}
\tilde\delta\doo &=  \doo \doo + \duo\esu\dos & \tilde\delta\ess &= \ess\ess + \eus\esu\\
\tilde\delta\dos &= \dos\doo + \ess\dos + \dus\esu\dos & \tilde\delta\esu &= \esu\ess + \euu\esu\\
\tilde\delta\duo&= \doo\duo + \duo\euu + \duo\esu\dus & \tilde\delta\eus &= \ess\eus + \eus\euu\\
\tilde\delta\dus &=  \eus + \dos\duo + \ess\dus + \dus\euu + \dus\esu\dus & \tilde\delta\euu &= \esu\eus + \euu\euu 
\end{align*}
Alternatively, we may view the right-hand sides as relations, since each counts points in the boundary of a 1-dimensional space (itself a union of products of trajectory spaces).  As in Figure \ref{fig:morsePaths}, the structure equation takes the form
\begin{align*}
\left[\begin{array}{rr}
\tilde\delta\doo & (\tilde\delta\duo)\esu + \duo(\tilde\delta\esu) \\
\tilde\delta\dos & \delta\ess + (\tilde\delta\dus)\esu + \dus(\tilde\delta\esu)
\end{array}\right] = \partial^2.
\end{align*}
Since every term on the left vanishes, $\partial$ is indeed a differential.  A similar argument applies to $\partial_{\partial M}$ and $\partial_{M,\partial M}$.

\subsection*{C. \ \ \, Relation to cell structures and CW-homology}

On a closed manifold, a Morse-Smale function whose critical values increase with index induces a cell structure with one $k$-cell for each index-$k$ critical point.  The $k$-cell is given by the unstable manifold $U_o$ of the critical point $o_k$.  The Morse complex then coincides with the chain complex $C^\text{cell}(M)$ underlying $\mathrm{CW}$-homology.

A similar relationship should hold for a manifold with boundary, motivated by the local picture of passing a critical level in Figure \ref{fig:handles}.  We suggest that a regular function whose critical values increase with index induces a cell structure with:
\begin{itemize}
\item a $k$-cell $U_o \subset M$ for each interior critical point $o_k$.
\item a $k$-cell $U_s \subset \partial M$ for each stable critical point $s_k$.
\item a $k$-cell $U_u \subset M$ and a $(k-1)$-cell $\partial U_u \subset \partial M$ for each unstable critical point $u_k$.
\end{itemize}
This should follow from the argument in the closed case adapted to the double $\tilde M$, though we do not attempt this here.

Assuming the above, for each unstable critical point, the $(k-1)$-cell $\partial U_u$ occurs with multiplicity $1$ in the boundary of the $k$-cell $U_u$.  There is therefore a single arrow between each pair of unstable generators in $C^\text{cell}_*(M)$ and a ``cancellation'' homotopy equivalence to a smaller complex $C'_*(M)$ generated by the interior and stable cells.  The components of the differential on $C'_*(M)$ running from stable to interior generators are induced by ``zig-zags'' in $C_*^\text{cell}(M)$ involving the algebraically-cancelled pairs of unstable generators.  A short exercise reveals that the complex $C'_*(M)$ is naturally identified with the Morse complex.  In Figure \ref{fig:cellMorse}, we illustrate this relationship for the example in Figure \ref{fig:morseExample}. A similar exercise applies to the relative and boundary complexes without cancellation.  In summary:
\begin{align*}
C_*^\text{cell}(M) \xrightarrow{\sim} C'_*(M) &= C^\text{morse}_*(M) \\
C_*^\text{cell}(M, \partial M) &= C^\text{morse}_*(M, \partial M) \\
C_*^\text{cell}(\partial M) &= C^\text{morse}_*(\partial M)
\end{align*}
In this way, the rather strange Morse differentials arises naturally from the cellular perspective.

\begin{figure}[htp]
\centering
\includegraphics[width=135mm]{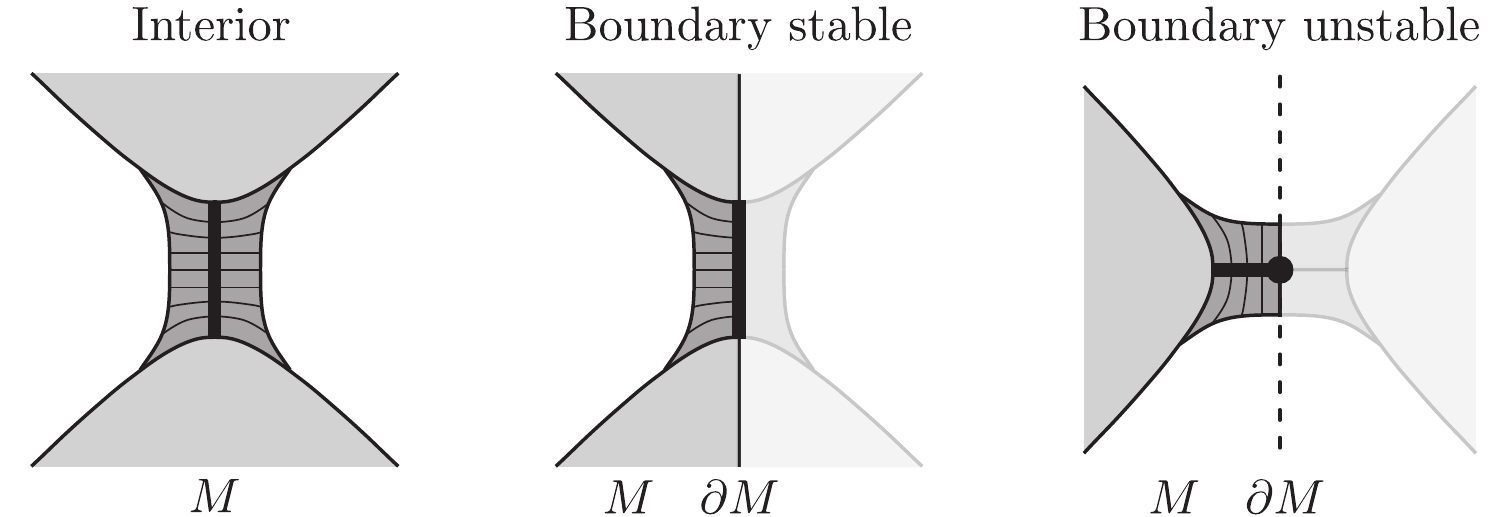}
\caption{The local level-set of $M$ just about a critical value for each type of critical point, with the mirror of $M$ faded.}
\label{fig:handles}
\end{figure}

\begin{figure}[htp]
\centering
\includegraphics[width=135mm]{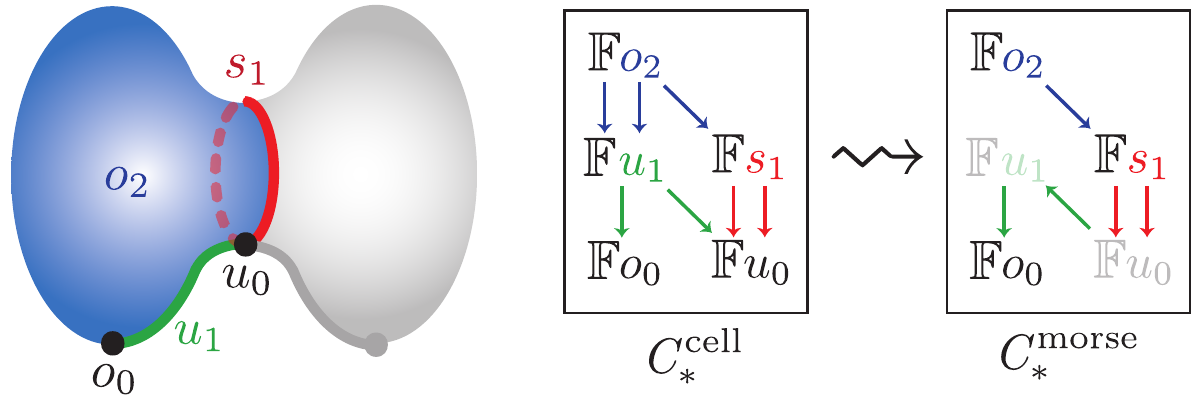}
\caption{We consider the example in Figure \ref{fig:morseExample} from the cellular perspective.  The cellular chain complex is generated by five cells and the differential encodes the algebraic multiplicities of the attaching maps.  The Morse complex results from algebraically canceling the pair of unstable cells.}
\label{fig:cellMorse}
\end{figure}

\begin{figure}[htp]
\centering
\includegraphics[width=156mm]{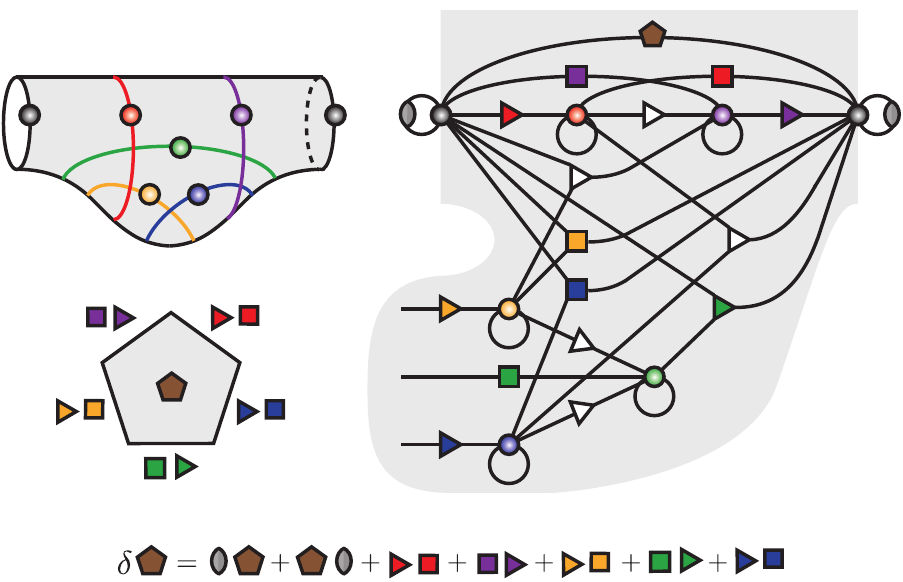}
\caption{The proof of the surgery exact triangle for $\check{\mathit{HM}}$ involves a map $\check{G}$ defined by counting monopoles on a cobordism over a family of metrics parameterized by a pentagon.  The edges of the pentagon correspond to five hypersurfaces which intersect as illustrated at top left (compare with Figure 2 in \cite{kmos}).  The map $\check{G}$ corresponds to the edge labelled by the brown pentagon in the higraph, which has 7 vertices and and 23 edges.  In \cite{kmos}, the four aligned vertices are type 2 ($\check{\mathit{HM}}$) and the remaining three vertices are type 1 ($\widehat{\mathit{HM}}$).  The path DGA on the bigraph (which need not be shown!) determines the terms of all monopole maps represented by edges in the higraph, and the structure equation implies that that $\check{G}$ satisfies the relation at the bottom of the figure (this agrees with Proposition 5.5 in \cite{kmos} because the yellow and blue triangles are zero maps for geometric reasons).  Note that since this higraph is tree-like, the assignment $\psi \equiv 1$ is balanced as well, leading to the surgery triangle for $\widehat{\mathit{HM}}$.  The key point here is that, due to Theorem \ref{thm:bistructure}, one can justifiably work on the level of higraphs, as though all monopoles in life were irreducible after all.}
\label{fig:exactTriangle}
\end{figure}

\bibliographystyle{hplain}
\bibliography{linksurg}


\noindent {\small \textsc{Department of Mathematics, Massachusetts Institute of Technology, Cambridge, USA}} \\
\textit{E-mail address:} \texttt{jbloom@math.mit.edu} \\
\textit{URL:} \url{http://math.mit.edu/~jbloom/}

\end{document}

%% file: macros.tex
\newcommand\hmb{\Hto_\bullet}
\newcommand\hmfrom{\widehat{\textit{HM}}_\bullet}
\newcommand\hmbar{\overline{\textit{HM}}_\bullet}

\newcommand\doo{\partial^o_o}
\newcommand\dos{\partial^o_s}
\newcommand\duo{\partial^u_o}
\newcommand\dus{\partial^u_s}

\newcommand\ess{\bar{\partial}^s_s}
\newcommand\esu{\bar{\partial}^s_u}
\newcommand\eus{\bar{\partial}^u_s}
\newcommand\euu{\bar{\partial}^u_u}

\newtheorem{theorem}{Theorem}[section]
\newtheorem{lemma}[theorem]{Lemma}

\newtheorem{corollary}[theorem]{Corollary}
\newtheorem{proposition}[theorem]{Proposition}

\theoremstyle{definition}
\newtheorem{definition}[theorem]{Definition}

\newtheorem{remark}[theorem]{Remark}